\documentclass{article}

\usepackage{amsmath}
\usepackage{amsthm}
\usepackage{amssymb}
\usepackage{mathtools}
\usepackage{tikz}
\usepackage{ytableau}
\usepackage{graphicx}
\usepackage{yhmath}
\usepackage{geometry}
\usepackage{marvosym}

\theoremstyle{plain}
\newtheorem{thm}{Theorem}[section]
\newtheorem{lem}[thm]{Lemma}
\newtheorem{prop}[thm]{Proposition}
\newtheorem{cor}[thm]{Corollary}

\theoremstyle{definition} 
\newtheorem{exmp}[thm]{Example}

\theoremstyle{remark}
\newtheorem{rem}[thm]{Remark}

\usepackage{cleveref} 

\crefname{thm}{Theorem}{Theorems}
\crefname{lem}{Lemma}{Lemmas}
\crefname{prop}{Proposition}{Propositions}
\crefname{cor}{Corollary}{Corollaries}
\crefname{exmp}{Example}{Examples}
\crefname{rem}{Remark}{Remarks}

\crefalias{thm}{thm}
\crefalias{lem}{lem}
\crefalias{prop}{prop}
\crefalias{cor}{cor}
\crefalias{exmp}{exmp}
\crefalias{rem}{rem}

\DeclareMathOperator{\Pfn}{\mathbf{Pfn}}
\DeclareMathOperator{\Set}{\bf{Set}}
\DeclareMathOperator{\PT}{PT}
\DeclareMathOperator{\PO}{PO}
\DeclareMathOperator{\id}{id}
\DeclareMathOperator{\dom}{dom}
\DeclareMathOperator{\im}{im}
\DeclareMathOperator{\Ext}{Ext}
\DeclareMathOperator{\Hom}{Hom}
\DeclareMathOperator{\GL}{GL}
\DeclareMathOperator{\E}{E}
\DeclareMathOperator{\SE}{SE}
\DeclareMathOperator{\EO}{EO}
\DeclareMathOperator{\SEO}{SEO}
\DeclareMathOperator{\Irr}{Irr}
\DeclareMathOperator{\IRep}{IRep}
\DeclareMathOperator{\tr}{tr}
\DeclareMathOperator{\sgn}{sgn}
\DeclareMathOperator{\trace}{trace}
\DeclareMathOperator{\Res}{Res}
\DeclareMathOperator{\Ind}{Ind}
\DeclareMathOperator{\Inf}{Inf}
\DeclareMathOperator{\Sym}{Sym}
\DeclareMathOperator{\Alt}{Alt}
\DeclareMathOperator{\dec}{dec}
\DeclareMathOperator{\pd}{pd}
\DeclareMathOperator{\gd}{glDim}
\DeclareMathOperator{\Rad}{Rad}

\newcommand{\Rt}{\widetilde{\mathcal{R}}_E}
\newcommand{\Rtm}{\widetilde{\mathcal{R}}_\mathcal{E}}
\newcommand{\Lt}{\widetilde{\mathcal{L}}_E}
\newcommand{\Ltm}{\widetilde{\mathcal{L}}_\mathcal{E}}

\title{The representation theory of the wreath product of a finite group with the monoid of all partial functions on a finite set as an EI-category algebra}
\author{Itamar Stein\thanks{Mathematics Unit, Shamoon College of Engineering, 77245 Ashdod, Israel}\\ \\ \Letter \, Steinita@gmail.com }

\date{}

\begin{document}

\maketitle

\begin{abstract}
Let $G$ be a finite group. We provide a description of the ordinary quiver of the complex monoid algebra of the wreath product $G \wr \PT_n$, where $\PT_n$ denotes the monoid of all partial functions on an $n$-element set. This description depends on the multiplicities of simple $G$-modules appearing in the decomposition of tensor products of simple $G$-modules. We also prove that the global dimension of this algebra is $n-1$. Both results are obtained by analyzing the associated Ehresmann EI-category related to the monoid. Finally, we describe the quiver of the algebra of the wreath product of $G$ with the submonoid of all order-preserving partial functions.
\end{abstract}

\section{Introduction}
A central aim in the study of monoid representations is to connect them with the general representation theory of associative algebras. Investigating the invariants of the monoid algebra of a finite monoid $M$ is of particular interest. Throughout this paper, all modules are over the field of complex numbers $\mathbb{C}$, so we consider the complex monoid algebra $\mathbb{C}M$. 

Among invariants of an algebra, the (ordinary) quiver plays a foundational 
role. A standard way to present a finite-dimensional algebra is 
via a quiver, which is a directed graph, bound by a set of relations 
on its paths. In this context, the quiver can be viewed as the generators 
part in a generators-and-relations presentation of the algebra. Even without explicit knowledge of the relations, the quiver alone encodes a significant amount of data about the algebra. 

Determining the quiver of a monoid algebra is a fundamental problem in the representation theory of finite monoids \cite[Chapter 17]{Steinberg2016}. To date, descriptions of the quiver have been obtained for 
many monoids and families of monoids
\cite{Denton2010,Margolis2015,Margolis2021B,Margolis2011,Margolis2012,Margolis2018B,Ringel2000,Saliola2007,Shahzamanian2021,Stein2016,Stein2026,Steinberg2025}.
 In particular, in \cite{Stein2016}, the author described the quiver of the monoid algebra $\mathbb{C}\PT_n$, where $\PT_n$ is the monoid of all partial functions on an $n$-element set. In this paper, we aim to generalize this result to the complex monoid algebra of the partial wreath product $G\wr \PT_n$ of any finite group $G$ with $\PT_n$.

 Partial wreath products of groups with (partial) transformation semigroups play a crucial role in the Krohn-Rhodes decomposition theory of finite automata and Krohn-Rhodes complexity theory \cite{qTheory2009}. This monoid was also studied recently in \cite[Section 9]{Gould2022}.

To compute the quiver of the monoid algebra of $G\wr \PT_n$, we follow the general framework established in \cite{Stein2016} for $\PT_n$. However, the introduction of the finite group $G$ into the structure introduces complications in the underlying representation theory.

After reviewing the necessary preliminaries in \Cref{sec:preliminaries}, we introduce several branching rules in \Cref{sec:branching_rules} that will be used in our proof.

The computation of the quiver is carried out in \Cref{sec:quiver_of_G_wr_PT_n}. Since $\mathbb{C}(G \wr \PT_n)$ is a finite $E$-Ehresmann and right restriction semigroup, its algebra is isomorphic to that of its associated Ehresmann category.

This category is the wreath product of $G$ with the category of all onto functions between subsets of an $n$-element set (see also \cite{Snowden2019}).
This is an EI-category, meaning that every endomorphism monoid is a group. To determine the quiver of an EI-category algebra, a well-known method (see \cite[Section 6.3.1]{Margolis2012} and \cite{Li2011}) reduces the problem to the representation theory of its endomorphism groups. We follow this approach here.

The endomorphism groups in this case are of the form $G \wr S_k$, where $S_k$ is the symmetric group. This computation utilizes known branching rules for $G \wr S_k$ to establish our main result.
The description of the quiver itself is given in \Cref{thm:Quiver_of_GwrE_n} and relies on the decomposition of tensor products of simple $G$-modules.

We remark that the similar problem of finding the quiver of the wreath product of $G$ with the category of injective functions between subsets of an $n$-element set was solved in \cite[Section 6]{Stein2017B}. However, in that case, the structure of the group $G$ plays a minimal role in the description.

In \Cref{sec:global_dimension_of_PT_n}, we prove that the global dimension of $\mathbb{C} (G\wr \PT_n)$ is $n-1$ by applying the EI-category framework to lift known results on the global dimension of $\PT_n$ from \cite{Stein2019}.

Finally, in \Cref{sec:quiver_of_PO_n}, we compute the quiver of the complex monoid algebra $\mathbb{C}(G\wr \PO_n)$, where $\PO_n$ is the monoid of all order-preserving partial functions on an $n$-element set. The method is identical to that applied in the case of $\PT_n$, but the branching rules are simpler in this setting.

\textbf{Acknowledgement:} The author is grateful to Professor Stuart Margolis for several discussions on the monoid $G\wr \PT_n$.

\section{Preliminaries} \label{sec:preliminaries}

\subsection{Partial wreath product}

Let $\mathcal{A}$ be a small category. We denote by $\mathcal{A}^{0}$
and $\mathcal{A}^{1}$ the sets of objects and morphisms of $\mathcal{A}$,
respectively. For $a,b\in\mathcal{A}^{0}$, we write $\mathcal{A}(a,b)$
for the hom-set of morphisms with domain $a$ and range $b$. Recall
that a monoid can be viewed as a category with a single object. Following
\cite{Harold2011}, we denote by $\Pfn$ the category whose objects
are finite sets and whose morphisms are partial functions, and by
$\Set$ the subcategory consisting of total functions as morphisms.

Let $G$ be a finite group with identity element $1_{G}$, and let $\PT(\mathcal{X},G)$ denote the set of all partial functions from $\mathcal{X}$ to $G$. For every $f \in \PT(\mathcal{X},G)$, we denote by $\dom(f)$ its domain. The set $\PT(\mathcal{X},G)$ is a monoid under pointwise multiplication, where for any $f_{1}, f_{2} \in \PT(\mathcal{X},G)$, the product $f_{1} \cdot f_{2}$ has domain $\dom(f_{1}) \cap \dom(f_{2})$ and is defined by:
\[
(f_{1} \cdot f_{2})(x) = f_{1}(x) \cdot f_{2}(x) \quad \text{for all } x \in \dom(f_{1}) \cap \dom(f_{2}).
\]
The identity element of this monoid is the constant function mapping every $x \in \mathcal{X}$ to $1_{G}$.

Let $H:\mathcal{A}\to\Pfn$ be a functor. Define a new category $G\wr_{H}\mathcal{A}$
in the following way. The set of objects is the same as the set of
objects of $\mathcal{A}$, that is, $(G\wr_{H}\mathcal{A})^{0}=\mathcal{A}^{0}$.
Given two objects $a,b\in\mathcal{A}^{0}$, the hom-set $(G\wr_{H}\mathcal{A})(a,b)$
is \[\{(f,m)\mid f\in\PT(H(a),G),\,m\in\mathcal{A}(a,b),\,\text{where}\,\dom(f)=\dom(H(m))\}.\]
Now, given two morphisms $(f,m)\in (G\wr_{H}\mathcal{A})(a,b)$ and
$(f^{\prime},m^{\prime})\in (G\wr_{H}\mathcal{A})(b,c)$ the composition
is 
\[
(f^{\prime},m^{\prime})\cdot(f,m)=((f^{\prime}\circ H(m))\cdot f,m^{\prime}m).
\]

It is routine to verify that this composition is well defined and
that $G\wr_{H}\mathcal{A}$ is indeed a category. If $\id_{a}$ is
the identity morphism of $a\in\mathcal{A}^{0}$ and ${\bf 1}_{H(a)}:H(a)\to G$
is the function defined by ${\bf 1}_{H(a)}(x)=1_{G}$ for every $x\in H(a)$,
then the identity morphism of the object $a\in(G\wr_{H}\mathcal{A})^{0}$
is $({\bf 1}_{H(a)},\id_{a})$. The category $G\wr_{H}\mathcal{A}$
is called the \emph{partial wreath product }of $G$ and $\mathcal{A}$
with respect to $H$. We will apply this construction in two special
cases. In the first case, $H$ is a functor $H:\mathcal{A}\to\Set$. In this case, the construction reduces to the standard wreath product of a group
with a category (see, for example, \cite{Snowden2019,Wells1980}).
In the second case, $\mathcal{A}$ is a monoid $M$ with identity
element $1_{M}$. Here the construction is well-known, even when $G$ and
$M$ are semigroups. It appears in \cite{Eilenberg1976} in the language
of transformation semigroups, and in \cite{Knauer1980} under the
name of 0-wreath products. See also \cite[Section 9.1]{Gould2022}
and references therein. In this case, a functor $H:M\to\Pfn$ is an
\emph{incomplete $M$-action}. That is, it consists of a set $\mathcal{X}$ and a monoid homomorphism 
$\varphi: M \to \PT_{\mathcal{X}}$. For $m \in M$ and $x \in \mathcal{X}$, 
we will write $m \bullet x$ instead of $\varphi(m)(x)$. For every 
$m \in M$, we write 
\[
\dom(m) = \dom(\varphi(m)) \subseteq \mathcal{X}.
\]
The monoid $M$ acts on the right of $\PT(\mathcal{X},G)$ by $f\ast m=f\circ\varphi(m)$.
Explicitly, for $x\in\mathcal{X}$, the partial function $f\ast m$
is given by 
\[
(f\ast m)(x)=\begin{cases}
f(m\bullet x) & m\bullet x\enspace\text{and}\enspace f(m\bullet x)\enspace\text{are both defined}\\
\text{undefined} & \text{otherwise}.
\end{cases}
\]
In this case, the partial wreath product $G\wr_{\mathcal{X}}M$ is
a monoid whose underlying set is 
\[
G\wr_{\mathcal{X}}M=\{(f,m)\mid m\in M,\quad f\in\PT(\mathcal{X},G),\quad\dom(f)=\dom(m)\}.
\]
The operation is given by 
\[
(f_{1},m_{1})\cdot(f_{2},m_{2})=((f_{1}\ast m_{2})\cdot f_{2},m_{1}m_{2}),
\]
where $(f_{1}\ast m_{2})(x)=f_1(m_2\bullet x)$
and the identity element is $({\bf 1}_{\mathcal{X}},1_{M})$.
\begin{rem}
Many authors adopt the convention of composing functions from left
to right. Under this convention, the monoid $M$ acts on the right
of $\mathcal{X}$ and on the left of $\PT(\mathcal{X},G)$. Consequently,
the multiplication in $G\wr_{\mathcal{X}}M$ is often written in a
different form in the literature.
\end{rem}

The case where $M$ is a group is of great importance. If $f:\mathcal{X}\to G$
is a (total) function, we denote by $f^{-1}:\mathcal{X}\to G$ the
function defined by $f^{-1}(x)=\left(f(x)\right)^{-1}$ for every
$x\in\mathcal{X}.$ If $M$ is a group, then $G\wr_{\mathcal{X}}M$
is a group, and the inverse of $(f,m)\in G\wr_{\mathcal{X}}M$ is
given by 
\[
(f,m)^{-1}=(\left(f\ast m^{-1}\right)^{-1},m^{-1}).
\]
There is a natural incomplete action of the monoid $\PT_{\mathcal{X}}$
on the set $\mathcal{X}$. Formally, the action is given by the identity
function $\varphi:\PT_{\mathcal{X}}\to\PT_{\mathcal{X}}$. In this
case, we simply write $G\wr\PT_{\mathcal{X}}$ for $G\wr_{\mathcal{X}}\PT_{\mathcal{X}}$.
If $\mathcal{X}=[n]=\{1,\ldots,n\}$, we denote the corresponding
wreath product by $G\wr\PT_{n}$. In this special case, the wreath
product has a natural description using matrices over $G\cup\{0\}$.
For $\alpha\in\PT_{n}$ and $f\in\PT([n],G)$ with $\dom(f)=\dom(\alpha)$,
we denote by $[f,\alpha]$ an $n\times n$ matrix over $G\cup\{0\}$
defined by 
\[
[f,\alpha]_{i,j}=\begin{cases}
f(j) & \alpha(j)=i,\\
0 & \text{otherwise}.
\end{cases}
\]
Let $\mathcal{M}$ denote the set of all such matrices. Note that
each $[f,\alpha]\in\mathcal{M}$ has at most one non-zero entry in
each column, so matrix multiplication is well-defined. It is straightforward
to verify that the function $\psi:G\wr\PT_{n}\to\mathcal{M}$ defined
by $\psi((f,\alpha))=[f,\alpha]$ is a monoid isomorphism.

\subsection{Ehresmann semigroups}
Let $S$ be a semigroup and let $E\subseteq S$ be a subset of idempotents.
We define two  equivalence relations  $\Lt$ and $\Rt$ on $S$.
\[
a\Lt b\iff(\forall e\in E\quad be=b\Leftrightarrow ae=a)
\]
\[
a\Rt b\iff(\forall e\in E\quad eb=b\Leftrightarrow ea=a).
\]
A subset $E\subseteq S$ of idempotents is called a \emph{subsemilattice} if it is a commutative subsemigroup. It is well known that any commutative semigroup of idempotents
has the structure of a semilattice (i.e. a poset where every two elements
have a meet) if one defines $a\leq b$ whenever $ab=ba=a$. A semigroup
$S$ with a subsemilattice $E\subseteq S$ is called \emph{right $E$-Ehresmann
}if every $\Lt$-class contains a unique idempotent from $E$ and
$\Lt$ is a right congruence. We denote the unique idempotent in the
$\Lt$-class of $a$ by $a^{\ast}$. Note that $a^{\ast}$ is the
unique minimal element $e\in E$ such that $ae=a$. It is well known
that $\Lt$ is a right congruence if and only if the identity $(ab)^{\ast}=(a^{\ast}b)^{\ast}$
holds for every $a,b\in S$.

Dually, we can consider semigroups for which every $\Rt$ class contains
a unique idempotent. We denote the unique idempotent in the $\Rt$
class of $a$ by $a^{+}$. Such a semigroup is called left $E$-Ehresmann
if $\Rt$ is a left congruence, or equivalently if $(ab)^{+}=(ab^{+})^{+}$
for every $a,b\in S$. A semigroup $S$ with a subsemilattice $E\subseteq S$
is called \emph{$E$-Ehresmann }if it is both left and right $E$-Ehresmann.
The semilattice $E$ is also called the set of \emph{projections }of
$S$.

A right (left) Ehresmann semigroup $S$ is called \emph{right (respectively,
left) restriction} if the ``right ample'' (respectively, ``left
ample'') identity $ea=a(ea)^{\ast}$ (respectively, $ae=(ae)^{+}a$)
holds for every $a\in S$ and $e\in E$.

For every Ehresmann semigroup $S$, we can associate a category ${\bf C}(S)$
in the following way. The object set of ${\bf C}(S)$ is the set $E$
of projections and the morphisms of ${\bf C}(S)$ are in one-to-one
correspondence with elements of $S$. For every $a\in S$, we associate 
a morphism $C(a)\in{\bf C}(S)^1$ with domain $a^{\ast}$ and range
$a^{+}$ and if the range of $C(a)$ is the domain of $C(b)$, the
composition $C(b)\cdot C(a)$ is defined to be $C(ba)$. For more
facts and proofs on Ehresmann semigroups and Ehresmann categories,
the reader is referred to \cite{Gould2010,Gould2010b}.

\subsection{Algebras and modules}

Let $A$ be an algebra. We will only discuss unital, finite-dimensional 
$\mathbb{C}$-algebras in this paper. Likewise, when we say that $V$ is a module 
over $A$ (or an $A$-module or an $A$-representation), we mean that 
$V$ is a finite-dimensional left module over $A$. For $a \in A$ and $v \in V$, 
we write $a \bullet v$ for the action of $a$ on the module element $v$. 
An $A$-module $V$ is \emph{simple} (or \emph{irreducible}) 
if $0$ is its only proper submodule.
The \emph{ordinary quiver} $Q$ of a finite dimensional
algebra $A$ is a directed graph defined in the following
way: The vertices of $Q$ are in a one-to-one correspondence with
the simple modules of $A$ (up to isomorphism). If $V_{i}$
and $V_{j}$ are simple modules of $A$ (identified with
two vertices of the quiver), then the number of arrows from $V_{i}$
to $V_{j}$ is
\[
\dim\Ext^{1}(V_{i},V_{j})
\]
where $\Ext^{1}(V,-)$ is the first right derived functor
of $\Hom(V,-)$, see \cite[Chapters 6-7]{Rotman2009}.
More about modules of algebras
and quivers can be found in \cite{Assem2006}.

In this paper, we discuss complex algebras of finite categories. Let $\mathcal{A}$ be a finite category. The \emph{category algebra} $\mathbb{C}\mathcal{A}$ is a vector space over $\mathbb{C}$ with the morphisms of $\mathcal{A}$  as its basis. It consists of all formal linear combinations:
\[
\left\{ \sum_{i=1}^{n} k_{i}m_{i} \mid k_{i} \in \mathbb{C},\, m_{i} \in \mathcal{A}^{1} \right\}.
\]
The multiplication in $\mathbb{C}\mathcal{A}$ is the linear extension of the following:
\[
m^{\prime} \cdot m = \begin{cases}
m^{\prime}m & \text{if } m^{\prime}m \text{ is defined in } \mathcal{A}, \\
0 & \text{otherwise}.
\end{cases}
\]
The algebra $\mathbb{C}\mathcal{A}$ has a unit element given by the sum of the identity morphisms of all objects in $\mathcal{A}$:
\[
1_{\mathbb{C}\mathcal{A}} = \sum_{c \in \mathcal{A}^0} \id_{c}.
\]

Since a monoid is a category with a single object, this construction naturally specializes to the definition of a \emph{monoid algebra}. In this case, the algebra $\mathbb{C}M$ consists of all formal linear combinations of elements of the monoid $M$, with multiplication defined as the linear extension of the monoid operation.

\subsection{Complex group representations}
Let $G$ be a finite group. If $V$ is a $\mathbb{C}$G-module, we will usually simply say that $V$ is a $G$-module (or a $G$-representation). Equivalently, a $G$-module is a pair $(V,\rho)$ of a $\mathbb{C}$-vector space $V$ and a group homomorphism $\rho:G\to\GL(V)$.
We denote the set of simple modules of $G$ (up to isomorphism) by $\IRep G$. It is well known that every $G$-module
is a finite direct sum of simple modules and that the
number of different simple $G$-modules (up to isomorphism)
is the number of conjugacy classes of $G$.
We denote the trivial
module of any group $G$ by $\tr_{G}$. Recall that if $V$
is a $G$-module, then $V^{\ast}=\Hom(V,\mathbb{C})$ is
also a $G$-module with operation $(g\bullet\varphi)(v)=\varphi(g^{-1}\bullet v)$.
Let $U$ and $V$ be $G$-modules. The inner tensor product
$U\otimes V$ is again a $G$-module with action defined by
$g\bullet(u\otimes v)=(g\bullet u)\otimes (g\bullet v)$ and extending linearly. Now, assume
that $U_{1}$ and $U_{2}$ are modules of $G_{1}$ and $G_{2}$,
respectively. The outer tensor product $U_{1}\otimes U_{2}$ of $U_{1}$
and $U_{2}$ is the ($G_{1}\times G_{2}$)-module where $(g_{1},g_{2})\bullet (u_{1}\otimes u_{2})=(g_{1}\bullet u_{1})\otimes(g_{2}\bullet u_{2})$.
Although the two types of tensor product can be distinguished by
context, we prefer to use a different notation for the outer tensor product,
denoting it by $\boxtimes$. Similarly, the simple tensors of $U\boxtimes V$
will be denoted by $u\boxtimes v$. It is well known that $\IRep(G_{1}\times G_{2})=\{U\boxtimes V\mid U\in\IRep G_{1},V\in\IRep G_{2}\}$.
The \emph{character $\chi_{U}$ }of the $G$-module $(U,\rho)$
is the function $\chi_{U}:G\to\mathbb{C}$ defined by $\chi_{U}(g)=\trace(\rho (g))$.
Recall that the multiplicity of $U\in\IRep G$ as a simple constituent
in some $G$-module $V$ is given by the inner product of characters
\[
\langle\chi_{U},\chi_{V}\rangle_G=\frac{1}{|G|}\sum_{g\in G}\chi_{U}(g)\overline{\chi_{V}(g)}.
\]
We may omit the subscript $G$ when the group is clear from the context.
Recall also that $\chi_{V^{\ast}}(g)=\overline{\chi_{V}(g)}$,
$\chi_{U\boxtimes V}((g_{1},g_{2}))=\chi_{U}(g_{1})\chi_{V}(g_{2})$ and $\chi_{U\otimes V}(g)=\chi_{U}(g)\chi_{V}(g)$.
In order to simplify notation, we will usually omit $\chi$ and
write $U$ also for the character of $U$. Hence the above inner product
will be written as
\[
\langle U,V\rangle=\frac{1}{|G|}\sum_{g\in G}U(g)\overline{V(g)}.
\]
Let $(U,\rho)$ be a $G$-module and let $H\leq G$ be a subgroup.
The \emph{restriction} of $(U,\rho)$ to $H$ is the $H$-module $(\Res_{H}^{G}U,\Res_{H}^{G}\rho)$
defined by 
\[
\Res_{H}^{G}\rho(h)(u)=\rho(h)(u)
\]
that is, restricting the homomorphism to the subgroup $H$. Note that
$\dim\Res_{H}^{G}U\allowbreak=\dim U$ and if $U$ is a simple
$G$-module, then $\Res_{H}^{G}U$ does not have to be a simple
$H$-module. Let $(U,\rho)$ be an $H$-module, the\emph{
induction} to $G$, denoted $(\Ind_{H}^{G}U,\Ind_{H}^{G}\rho)$ is
the tensor product 
\[
\Ind_{H}^{G}U=\mathbb{C}G\underset{\mathbb{C}H}{\otimes}U
\]
where the $G$ action is given by 
\[
g\bullet(s\otimes u)=(gs)\otimes u
\]
where $s\in\mathbb{C}G$ and $u\in U$. However, we will also use
the following more concrete description. Choose $S=\{s_{1},\ldots,s_{l}\}$
to be representatives of the left cosets of $H$ in $G$ (where $l=[G:H]$).
Note that any element $g\in G$ can be written in a unique way as
$g=s_{i}h$ where $s_{i}\in S$ and $h\in H$. Every element of $\Ind_{H}^{G}U$
is a formal sum of the form

\[
\alpha_{1}(s_{1},u_{1})+\ldots+\alpha_{l}(s_{l},u_{l})
\]
where $u_{i}\in U$ and $\alpha_{i}\in\mathbb{C}$. In other words,
as a vector space $\Ind_{H}^{G}U$ is ${\displaystyle \bigoplus_{i=1}^{l}U}$,
that is, $l$ copies of $U$. The action is defined on elements of
the form $(s_{i},u)$ by 
\[
g\bullet (s_{i},u)=(s_{j},h\bullet u)
\]
where $s_{j}$ and $h$ are unique such that $gs_{i}=s_{j}h$. The
required action is given by extending linearly. Note that $\dim\Ind_{H}^{G}U=[G:H]\dim U$.
It is important to mention that the modules $\Ind_{H}^{G}U$
and $\Res_{H}^{G}V$ depend not only on the groups $G$ and $H$
but also on the specific embedding of $H$ into $G$. Hence, we will
have to give the specific embeddings when discussing these modules.
Both induction and restriction are transitive and additive, that is,
if $K\leq H\leq G$ then
\[
\Ind_{H}^{G}\Ind_{K}^{H}U\cong\Ind_{K}^{G}U,\quad\Ind_{H}^{G}(U\oplus V)\cong\Ind_{H}^{G}U\oplus\Ind_{H}^{G}V
\]
and 
\[
\Res_{K}^{H}\Res_{H}^{G}U\cong\Res_{K}^{G}U,\quad\Res_{H}^{G}(U\oplus V)\cong\Res_{H}^{G}U\oplus\Res_{H}^{G}V.
\]
For the restriction this is a trivial statement, and for the induction the
proof is \cite[Propositions 1.1.10 and 1.1.11]{Tullio2014}. An important
fact that relates induction to restriction is the following one (for a
proof, see \cite[Corollary 1.1.20 ]{Tullio2014}). 
\begin{thm}[Frobenius reciprocity]
\label{thm:FrobeniusReciprocity} Let $H\leq G$ and let $U$ and
$V$ be $G$ and $H$-modules respectively. Then the multiplicity
of $V$ in $\Res_{H}^{G}U$ equals the multiplicity of $U$ in $\Ind_{H}^{G}V$. 
\end{thm}
Using characters, Frobenius reciprocity can be written as the following
equality
\[
\langle\Ind_{H}^{G}V,U\rangle_G=\langle V,\Res_{H}^{G}U\rangle_H.
\]

Let $U$ be a $G$-module. Consider the \emph{swap transformation} $S: U \otimes U \to U \otimes U$ defined on simple tensors by $S(u_1 \otimes u_2) = u_2 \otimes u_1$. We define the \emph{symmetric square} $\Sym^2 U$ and the \emph{alternating square} $\Alt^2 U$ as the following submodules of the tensor product $U \otimes U$:
\begin{align*}
    \Sym^2 U &= \{ v \in U \otimes U \mid S(v) = v \}, \\
    \Alt^2 U &= \{ v \in U \otimes U \mid S(v) = -v \}.
\end{align*}
As $G$-modules, $U \otimes U \cong \Sym^2 U \oplus \Alt^2 U$. The characters of these modules are given by:
\[ (\Sym^2 U)(g) = \frac{1}{2}\left(U(g)^2 + U(g^2)\right) \quad \text{and} \quad (\Alt^2 U)(g) = \frac{1}{2}\left(U(g)^2 - U(g^2)\right). \]

\subsection{Representation theory of $S_n$ and $G \wr S_n$}
We will recall some elementary facts regarding the representation
theory of the symmetric group. More details can be found in \cite{James1981,Sagan2001}.
Recall that an \emph{integer composition} of $n$ is a tuple $\lambda=[\lambda_{1},\ldots,\lambda_{k}]$
of non-negative integers such that $\lambda_{1}+\cdots+\lambda_{k}=n$
while an \emph{integer partition} of $n$ (denoted $\lambda\vdash n$)
is an integer composition such that $\lambda_{1}\geq\lambda_{2}\geq\cdots\geq\lambda_{k}>0$.
From now on, when dealing with a partition $\lambda$ we will write
its elements in superscript $\lambda=[\lambda^{1},\ldots,\lambda^{k}]$
because we want to reserve the subscript for multipartitions. Note
that $0$ has one partition, namely the empty partition, denoted by
$\emptyset$. We can associate to any partition $\lambda$ a graphical
description called a \emph{Young diagram}, which is a table with $\lambda^{i}$
boxes in its $i$-th row. For instance, the Young diagram associated
to the partition $[3,3,2,1]$ of $9$ is: 
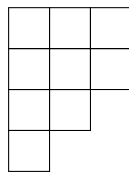
\[
\ydiagram{3,3,2,1}
\]

We will identify the two notions and regard integer partition and
Young diagram as synonyms. It is well-known that simple modules
of $S_{n}$ are indexed by integer partitions of $n$. We denote the
simple module associated to the partition $\lambda$
(also called its \emph{Specht module}) by $S^{\lambda}$. An explicit
description of $S^{\lambda}$ can be found in \cite[Section 2.3]{Sagan2001}.
It will be often convenient to draw the diagram $\lambda$ instead
of writing $S^{\lambda}$. For instance we may write

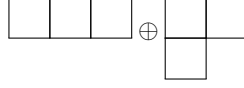
\[
\ydiagram{3}\oplus\ydiagram{2,1}
\]
instead of: $S^{\lambda}\oplus S^{\delta}$ for partitions $\lambda=[3]$
and $\delta=[2,1]$.

Let ${\bf n} = [n_1, \ldots, n_l]$ be an integer composition of $n$. A tuple $\Lambda = (\lambda_{1}, \ldots, \lambda_{l})$ such that $\lambda_{i} \vdash n_{i}$ for every $i$ is called a \emph{multipartition} of $n$ with $l$ components. We also call it a multipartition of the composition ${\bf n}$ and denote this by $\Lambda \Vdash {\bf n}$. We define a \emph{multi-Young diagram} to be a tuple of Young diagrams. As we identify partitions with Young diagrams, we also identify multipartitions with multi-Young diagrams. 

Let $G$ be a finite group with $l$ conjugacy classes. It is well-known (\cite[Theorem 2.6.1]{Tullio2014}) that multi-Young diagrams with $n$ boxes and $l$ components index the simple modules of the wreath product $G \wr S_{n}$. If $\Lambda \Vdash {\bf n}$ is a multi-Young diagram, then we denote by $S^\Lambda$ its associated $G \wr S_{n}$-module.

\section{Branching rules} \label{sec:branching_rules}
Let $G$ be a finite group with $l$ conjugacy classes. Fix $\IRep G=\{U_1,\ldots U_l\}$ to be its set of simple modules. In this section we will describe several known branching rules for the group $G\wr S_n$ that we will need later on.
\subsection{Restriction of $G \wr S_2$}
 Let $S_2=\{\id,s\}$ be the symmetric group of order $2$  where $s$ will always be the swap permutation. We denote by $\tr_2$ the trivial module of $S_2$ and by $\sgn_2$ the sign module of $S_2$. 
The group $G\times S_2$ can be embedded in $G \wr S_2$ by $\varphi(g,\sigma)=((g,g),\sigma)$. In this section, we study the restriction $\Res_{G\times S_2}^{G\wr S_2}$.
Every simple module  of $G\times S_2$ is of the form  $U_r \boxtimes \tr_2$ or $U_r \boxtimes \sgn_2$ where $U_r\in \IRep G$ ($r\in \{1,\ldots l\}$). 
The description of simple modules of $G \wr S_2$ is more complicated, but well-known (see \cite[Chapter 2]{Tullio2014} for the description of simple modules of $G \wr S_n$ in general).
If $U_i,U_j\in \IRep G$ are two non-isomorphic simple modules of $G$ then $U_i\boxtimes U_j$ is a $G\times G$ simple module. Clearly, $G\times G$ embeds in $G \wr S_2$ by $\varphi (g_1,g_2)=((g_1,g_2),\id)$. One type of simple $G \wr S_2$ module is obtained by induction
$W_{i,j}=\Ind_{G\times G}^{G\wr S_2}(U_i \boxtimes U_j)$. This gives us $\binom{l}{2}$ simple modules. Note that $\dim W_{i,j}=2\dim U_i \dim U_j$. The module $W_{i,j}$ corresponds  to the multipartition $\Lambda$ with $\lambda_i=\lambda_j=[1]$ and $\lambda_k=\emptyset$ for $k\neq i,j$.
Given $U_i\in \IRep G$, another type of simple module is the tensor product $U_i \boxtimes U_i$ where the action is \[((g_1,g_2),\sigma)\bullet (u_1\boxtimes u_2)= (g_{\sigma(1)}\bullet u_{\sigma(1)})\boxtimes (g_{\sigma(2)}\bullet u_{\sigma(2)}).\] We denote this simple module by $W_i^+$. This gives another $l$ simple modules. The module $W_i^+$ corresponds to the multipartition $\Lambda$ with $\lambda_i=[2]$ and $\lambda_k=\emptyset$ for $k\neq i$. 
Finally, we denote by $\Inf(\sgn_2)$ the \emph{inflation} of $\sgn_2$ to a $G \wr S_2$ module. This means that $((g_1,g_2),\sigma)$ acts like $\sigma$. The last $l$ simple modules of $G \wr S_2$ are obtained by the tensor product $W_i^+ \otimes \Inf(\sgn_2)$ and we denote them by $W_i^-$. The module $W_i^-$ corresponds to the multipartition $\Lambda$ with $\lambda_i=[1,1]$ and $\lambda_k=\emptyset$ for $k\neq i$. In total, we have $\frac{l^2+3l}{2}$ simple modules for $G\wr S_2$.
\begin{rem}
In the literature, the module $W_i^+$ is defined by the action
\[((g_1,g_2),\sigma)\bullet (u_1\boxtimes u_2)= (g_1\bullet u_{\sigma(1)})\boxtimes (g_2\bullet u_{\sigma(2)})\]
(in fact, with $\sigma^{-1}$, but $\sigma^{-1}=\sigma$ in our case.)
This variation arises because, in the context of groups, the composition in the wreath product $G \wr M$ is defined by \[(f_2,m_2)\cdot (f_1,m_1)=(f_2\cdot (m_2\star f_1),m_2m_1)\] where $(m\star f_1)(x)=f_1(m^{-1}x)$. As shown in \cite[Lemma 6.3]{Stein2017B}, the two definitions are isomorphic under the map
\[T(f,m)=(m\star f,m)\]
but it changes the concrete description of the module $W_i^+$.
In any case, we are interested in the action of elements of the form $((g,g),\sigma)$. For such elements, the two actions coincide, ensuring there is no ambiguity in our context.\end{rem}

The characters of these modules are well-known and essential for computing the desired restriction. Recall that we use the same notation for a module and its character. If we take $((1_G,1_G),\id)$ and $((1_G,1_G),s)$ as representatives of the cosets of $G\times G$ in $G\wr S_2$, it is easy to see that $((g,g), s)$ swaps the cosets and $((g,g), \id)$ fixes them. The character of the induction is summation of the base character over the fixed cosets (see \cite[Section 1.12]{Sagan2001}). Therefore, the character $W_{i,j}$ for an element $((g,g), \sigma) \in G \wr S_2$ is given by:
\[ W_{i,j}((g,g), \sigma) = 
\begin{cases} 
2 U_i(g) U_j(g) & \text{if } \sigma = \id, \\
0 & \text{if } \sigma = s.
\end{cases} \]
For the modules $W_i^+$ and $W_i^-$, the character values on the elements $((g,g), \sigma)$ are given as follows (see \cite[4.3.10 (vi), p.~150]{James1981}):
\[ W_i^{+}((g,g), \sigma) = 
\begin{cases} 
U_i(g)^2 & \text{if } \sigma = \id, \\
U_i(g^2) & \text{if } \sigma = s,
\end{cases} \]
and
\[ W_i^{-}((g,g), \sigma) = 
\begin{cases} 
U_i(g)^2 & \text{if } \sigma = \id, \\
-U_i(g^2) & \text{if } \sigma = s.
\end{cases} \]

\begin{lem} \label[lemma]{lem:Induction_GwrS_2Wij}
The multiplicities of $U_k \boxtimes \tr_2$ and $U_k \boxtimes \sgn_2$ in $\Res_{G \times S_2}^{G \wr S_2} W_{i,j}$ are both equal to the multiplicity of $U_k$ in $U_i \otimes U_j$.
\end{lem}

\begin{proof}
Let $V$ be either the trivial module $\tr_2$ or the sign module $\sgn_2$ of $S_2$. The multiplicity of $U_k \boxtimes V$ in the restriction is:
\[ \langle \Res_{G \times S_2}^{G \wr S_2} W_{i,j}, U_k \boxtimes V \rangle_{G \times S_2} = \frac{1}{2|G|} \sum_{g \in G} \sum_{\sigma \in S_2} W_{i,j}((g,g), \sigma) \overline{U_k(g) V(\sigma)}. \]
As $W_{i,j}((g,g), s) = 0$, the sum over $\sigma \in S_2$ only has a contribution from $\sigma = \id$. Since $V(\id) = 1$ for both the trivial and sign modules, the expression becomes:
\[ \frac{1}{2|G|} \sum_{g \in G} W_{i,j}((g,g), \id) \overline{U_k(g)} = \frac{1}{2|G|} \sum_{g \in G} 2 U_i(g) U_j(g) \overline{U_k(g)} = \langle U_i \otimes U_j, U_k \rangle_G, \]
which completes the proof.
\end{proof}

\begin{lem} \label[lemma]{lem:Induction_GwrS_2}
The multiplicities of $U_k \boxtimes \tr_2$ and $U_k \boxtimes \sgn_2$ in the restrictions $\Res_{G \times S_2}^{G \wr S_2} W_i^+$ and $\Res_{G \times S_2}^{G \wr S_2} W_i^-$ are given by the multiplicities of $U_k$ in $\Sym^2 U_i$ and $\Alt^2 U_i$ as follows:
\begin{enumerate}
    \item $\langle \Res_{G \times S_2}^{G \wr S_2}W_i^+, U_k \boxtimes \tr_2 \rangle_{G \times S_2} = \langle \Sym^2 U_i, U_k \rangle_{G}$
    \item $\langle \Res_{G \times S_2}^{G \wr S_2}W_i^+, U_k \boxtimes \sgn_2 \rangle_{G \times S_2} = \langle \Alt^2 U_i, U_k \rangle_{G}$
    \item $\langle \Res_{G \times S_2}^{G \wr S_2}W_i^-, U_k \boxtimes \tr_2 \rangle_{G \times S_2} = \langle \Alt^2 U_i, U_k \rangle_{G}$
    \item $\langle \Res_{G \times S_2}^{G \wr S_2}W_i^-, U_k \boxtimes \sgn_2 \rangle_{G \times S_2} = \langle \Sym^2 U_i, U_k \rangle_{G}$
\end{enumerate}
\end{lem}

\begin{proof}
We prove the first case; the others follow by similar character computations. By the definition of the inner product of characters on $G \times S_2$, we have:
\[ \langle \Res_{G \times S_2}^{G \wr S_2}W_i^+, U_k \boxtimes \tr_2 \rangle_{G \times S_2} = \frac{1}{2|G|} \sum_{g \in G} \left( W_i^+((g,g), \id) \overline{U_k(g)} + W_i^+((g,g), s) \overline{U_k(g)} \right). \]
Substituting the character values $W_i^+((g,g), \id) = U_i(g)^2$ and $W_i^+((g,g), s) = U_i(g^2)$, we obtain:
\[ \frac{1}{|G|} \sum_{g \in G} \left( \frac{U_i(g)^2 + U_i(g^2)}{2} \right) \overline{U_k(g)}. \]
The term in parentheses is precisely the character of $\Sym^2 U_i$. Thus, the expression reduces to $\langle \Sym^2 U_i, U_k \rangle_G$, which is the multiplicity of $U_k$ in $\Sym^2 U_i$.
\end{proof}

\subsection{Littlewood-Richardson rules for small additions}
Let $G$ be a group with $l$ conjugacy classes. The group $(G\wr S_k)\times (G\wr S_r)$ is naturally embedded in $G \wr S_{k+r}$. If $f_1:\{1,\ldots,k\}\to G$ and $f_2:\{1,\ldots,r\}\to G$ we define $f:\{1,\ldots,k+r\}\to G$ by
\[
f(i) = \begin{cases} 
f_1(i) & i \leq k \\ 
f_2(i-k) & i > k
\end{cases}.
\]
Then, the natural embedding $\varphi : G\wr S_{k}\times G\wr S_r\to G\wr S_{k+r}$ is defined by $\varphi((f_1,\sigma_1),(f_2,\sigma_2))=(f,\sigma_1\sigma_2)$ where $\sigma_1$ ($\sigma_2$) can be regarded as an element of $S_{k+r}$ that fixes $\{k+1,\ldots, k+r\}$ ($\{1,\ldots, k\}$).
The branching rules for describing the induction from $G\wr S_{k}\times G\wr S_r$ to $G\wr S_{k+r}$ are known (see \cite[Theorem 4.7]{Ingram2009} or \cite[Theorem 4.5]{Stein2017B}), but we will give here only the cases for $r=1,2$ as this is what we need in this paper.

Let ${\bf k}=[k_1,\ldots,k_l]$ be a composition of $k$ and let $\Lambda\Vdash{\bf k}$ be a multi-Young diagram and let $S^\Lambda$ be the associated $G\wr S_k$-module. Let $U_i$ be a $G=G\wr S_1$-module. We can think of it as a multi-Young diagram  with one box in the $i$-th component and all the other components are empty. Let $Y^+_i(\Lambda)$ be the set of multi-Young diagrams that can be obtained from $\Lambda$ by adding one box at the $i$-th component. Conversely, let $Y^-_i(\Lambda)$ be the set of multi-Young diagrams obtained from $\Lambda$ by removing one box from the $i$-th component.

\begin{prop} \label[proposition]{prop:Simple_Branching}
The induction and restriction rules are as follows:
\begin{enumerate}
    \item \textbf{Induction:} 
    \[ \Ind_{(G\wr S_k)\times G}^{G\wr S_{k+1}} (S^\Lambda\boxtimes U_i)=\bigoplus_{\Gamma\in Y^+_i(\Lambda)}S^{\Gamma} \]
    \item \textbf{Restriction:} By Frobenius reciprocity, the restriction of a $G\wr S_{k+1}$-module $S^\Gamma$ to the subgroup $(G\wr S_k)\times G$ is given by:
    \[ \Res^{G\wr S_{k+1}}_{(G\wr S_k)\times G} (S^\Gamma) = \bigoplus_{i=1}^{l} \bigoplus_{\Lambda \in Y^-_i(\Gamma)} S^\Lambda \boxtimes U_i \]
\end{enumerate}
\end{prop}

For the case $r=2$, we define the sets of multi-Young diagrams obtained by adding exactly two boxes to $\Lambda$:
\begin{itemize}
    \item $Y^+_{i,j}(\Lambda)$: one box added to component $i$ and one box added to component $j$ ($i \neq j$).
    \item $Y^+_{i,H^2}(\Lambda)$: a \textit{horizontal strip} of two boxes added to component $i$. This means that we cannot add the two boxes in the same column.
    \item $Y^+_{i,V^2}(\Lambda)$: a \textit{vertical strip} of two boxes added to component $i$. This means that we cannot add the two boxes in the same row.
\end{itemize}
\begin{prop} \label[proposition]{Inductopn_two_steps}
The induction from $(G \wr S_k) \times (G \wr S_2)$ to $G \wr S_{k+2}$ for the various simple $G \wr S_2$-modules is given by:
\[ \Ind_{(G\wr S_k)\times (G \wr S_2)}^{G\wr S_{k+2}} (S^\Lambda \boxtimes W_{i,j}) = \bigoplus_{\Gamma \in Y^+_{i,j}(\Lambda)} S^{\Gamma} \]
\[ \Ind_{(G\wr S_k)\times (G \wr S_2)}^{G\wr S_{k+2}} (S^\Lambda \boxtimes W_i^+) = \bigoplus_{\Gamma \in Y^+_{i,H^2}(\Lambda)} S^{\Gamma} \]
\[ \Ind_{(G\wr S_k)\times (G \wr S_2)}^{G\wr S_{k+2}} (S^\Lambda \boxtimes W_i^-) = \bigoplus_{\Gamma \in Y^+_{i,V^2}(\Lambda)} S^{\Gamma} \]
\end{prop}

\section{The quiver of $\mathbb{C}(G \wr \PT_n)$} \label{sec:quiver_of_G_wr_PT_n}
\subsection{Ehresmann structure} \label{sec:Ehresmann}
For every $X \subseteq [n]=\{1,\ldots n\}$ we define $\id_X$ to be the partial identity of the set $X$
\[ \id_X(x) = 
\begin{cases} 
x & \text{if } x\in X, \\
\text{undefined} & \text{otherwise,}
\end{cases} \]
and ${\bf 1}_X:[n]\to G$ to be the function
\[ {\bf 1}_X(x) = 
\begin{cases} 
1_G & \text{if } x\in X, \\
\text{undefined} & \text{otherwise.}
\end{cases} \]
Set $\mathcal{E}=\{({\bf 1}_X, \id_X)\mid X\subseteq [n]\}$ and note that this is a subsemilattice of $G \wr \PT_n$. It is routine to verify that the idempotent $(\mathbf{1}_{X}, \id_{X})$ is a left (right) identity of $(f, \alpha) \in G \wr \PT_{n}$ if and only if $\im(\alpha) \subseteq X$ (respectively, $\dom(\alpha)\subseteq X$). Therefore, we have that two elements $(f, \alpha)$ and $(g, \beta)$ of $G \wr \PT_n$ are $\Ltm$-related if and only if $\dom(\alpha) = \dom(\beta)$, and they are $\Rtm$-related if and only if $\im(\alpha) = \im(\beta)$. In other words, if $X=\dom(\alpha)$ and $Y=\im(\alpha)$ then
\[(f,\alpha)^\ast=({\bf 1}_X,\id_X),\quad(f,\alpha)^+=({\bf 1}_Y,\id_Y).\]
In \cite[Proposition 9.11]{Gould2022}, it was proved that  $G \wr \PT_n$ is a right $\mathcal{E}$-restriction monoid. It is also easy to prove that the left congruence condition holds.
\begin{lem}
The relation $\Rtm$ is a left congruence on $G \wr \PT_n$.
\end{lem}

\begin{proof}
Let $(f_1, \alpha_1), (f_2, \alpha_2) \in G \wr \PT_n$ such that $(f_1, \alpha_1) \Rtm (f_2, \alpha_2)$. This is equivalent to $\im(\alpha_1) = \im(\alpha_2)$. Let $(f_3, \alpha_3)$ be any element in $G \wr \PT_n$. Note that $\im(\alpha_3\alpha_1) = \im(\alpha_3\alpha_2)$ so
\[ (f_3, \alpha_3)(f_1, \alpha_1)=((f_{3}\ast\alpha_{1})\cdot f_{1},\alpha_{3}\alpha_{1}) \tilde{\mathcal{R}}_{\mathcal{E}}((f_{3}\ast\alpha_{2})\cdot f_{2},\alpha_{3}\alpha_{2})= (f_3, \alpha_3)(f_2, \alpha_2). \]
Thus, $\Rtm$ is a left congruence.
\end{proof}
It follows that $G \wr \PT _n$ is a $\mathcal{E}$-Ehresmann and right restriction monoid. This will be crucial in view of the following fact.
\begin{thm}[{\cite[Theorem 1.5]{Stein2018erratum}}] \label{thm:Ehresmann_Iso}
Let $M$ be a finite $E$-Ehresmann and right restriction monoid and let ${\bf C}(M)$ be its associated Ehresmann category. Then, for every unital commutative ring $\Bbbk$ there is an isomorphism of algebras $\Bbbk M\simeq \Bbbk {\bf C}(M)$
\end{thm}

Therefore, we can switch to studying the representation theory of the associated category \mbox{${\bf C}(G\wr \PT_n)$}. We start by describing it.
The objects of $\mathbf{C}(G \wr \PT_n)$ are in one-to-one correspondence with the elements of $\mathcal{E}$. Thus, the objects are of the form $(\mathbf{1}_X, \id_X)$ for $X \subseteq [n]$. For two subsets $X, Y \subseteq [n]$, the hom-set $\mathbf{C}(G \wr \PT_n)((\mathbf{1}_X, \id_X), (\mathbf{1}_Y, \id_Y))$ is identified with the elements $(f, \alpha) \in G \wr \PT_n$ such that 
\[ (f, \alpha)^* = (\mathbf{1}_X, \id_X) \quad \text{and} \quad (f, \alpha)^+ = (\mathbf{1}_Y, \id_Y). \]
Note that in this case, $X = \dom(\alpha)$ and $Y = \im(\alpha)$. We denote by $C(f, \alpha)$ the morphism associated with $(f, \alpha)$. 

Let $\E_{n}$ be the category defined as follows. The objects of $\E_{n}$ are subsets $X \subseteq [n]$. For $X, Y \subseteq [n]$, the hom-set $\E_{n}(X, Y)$ contains all the onto (total) functions $\alpha \colon X \to Y$. Let $G$ be a group. We denote by $G \wr \E_{n}$ the wreath product $G \wr_H \E_{n}$, where $H \colon \E_{n} \to \Set$ is the inclusion functor.

\begin{prop}
There is an isomorphism of categories $\mathbf{C}(G \wr \PT_{n}) \simeq G \wr \E_{n}$.
\end{prop}

\begin{proof}
It follows immediately from the above discussion. Formally, an isomorphism \[\psi \colon \mathbf{C}(G \wr \PT_{n}) \to G \wr \E_{n}\] is defined by $\psi((\mathbf{1}_{X}, \id_{X})) = X$ and $\psi(C(f, \alpha)) = (f, \alpha)$. Note that if $C(f, \alpha)$ is a morphism in $\mathbf{C}(G \wr \PT_{n})((\mathbf{1}_{X}, \id_{X}), (\mathbf{1}_{Y}, \id_{Y}))$, then $\dom(\alpha) = X$ and $\im(\alpha) = Y$, so $\alpha$ is indeed a total onto function $\alpha \colon X \to Y$. Moreover, $\dom(f) = X$, so $f \in G^{X}$ as required in the definition of $G \wr \E_{n}$. It is easy to see now that $\psi$ is an isomorphism.
\end{proof}

\subsection{The skeleton}\label{sec:The_Skeleton}
If $\mathcal{C}$ and $\mathcal{D}$ are equivalent categories, then their algebras are Morita equivalent (see \cite[Proposition 2.2]{Webb2007}). Since the quiver of an algebra is an invariant of Morita equivalence, we can switch our attention to a simpler category which is equivalent to $G \wr \E_{n}$. 

We can take a full subcategory with one object from every isomorphism class in $G \wr \E_{n}$. This category is called the \textit{skeleton} of $G \wr \E_n$. To describe it, we first have to characterize which objects in $G \wr \E_{n}$ are isomorphic.
\begin{lem}\label[lemma]{lem:Isomorphis_objects}
Two objects $X, Y$ in $G \wr \E_n$ are isomorphic if and only if $|X| = |Y|$.
\end{lem}

\begin{proof}
First, note that if $|X| < |Y|$, then the hom-set $(G \wr \E_n)(X, Y)$ is empty because there are no onto functions from $X$ to $Y$. If $X$ and $Y$ are isomorphic, then both $(G \wr \E_n)(X, Y)$ and $(G \wr \E_n)(Y, X)$ are non-empty, which implies $|X| = |Y|$.

Conversely, if $|X| = |Y|$, we can take any invertible function $\alpha \colon X \to Y$, and we claim that $(\mathbf{1}_X, \alpha)$ is an isomorphism with inverse $(\mathbf{1}_Y, \alpha^{-1})$. Indeed,
\begin{align*}
(\mathbf{1}_Y, \alpha^{-1}) \cdot (\mathbf{1}_X, \alpha) &= ((\mathbf{1}_Y \ast \alpha) \cdot \mathbf{1}_X, \alpha^{-1}\alpha) \\
&= ((\mathbf{1}_Y \ast \alpha) \cdot \mathbf{1}_X, \id_{X}).
\end{align*}
Note that since $\im(\alpha) = Y$, we have $\mathbf{1}_Y \ast \alpha = \mathbf{1}_X$. Therefore,
\[ (\mathbf{1}_Y, \alpha^{-1}) \cdot (\mathbf{1}_X, \alpha) = (\mathbf{1}_X \cdot \mathbf{1}_X, \id_{X}) = (\mathbf{1}_X, \id_{X}), \]
which is the identity morphism of the object $X$. Likewise, since $\alpha\alpha^{-1} = \id_{Y}$, it follows that $(\mathbf{1}_X, \alpha) \cdot (\mathbf{1}_Y, \alpha^{-1}) = (\mathbf{1}_Y, \id_{Y})$.
\end{proof}

Denote by $G \wr \SE_{n}$ the full subcategory of $G \wr \E_{n}$ whose objects are the sets $[0] = \emptyset$ and  $[k] = \{1, \dots, k\}$ for $1 \leq k \leq n$. As the notation suggests, this category can be identified with the wreath product of $G$ with the category $\SE_{n}$, where the set of objects is $\{[k] \mid 0 \leq k \leq n\}$ and the morphisms are total onto functions. Following the discussion above, to determine the quiver of the original algebra, it suffices to focus our attention on the quiver of the algebra $\mathbb{C}(G \wr \SE_n)$. 

It is also convenient to describe $G \wr \SE_{n}$ using matrices. Let $\mathcal{D}_{n}$ be the category defined as follows. The set of objects of $\mathcal{D}_{n}$ is $\{[k] \mid 0 \leq k \leq n\}$. For $0 \leq k, r \leq n$, the hom-set $\mathcal{D}_{n}([k], [r])$ consists of all $r \times k$ matrices over $G \cup \{0\}$ with exactly one non-zero element in every column and at least one non-zero element in every row. Composition of morphisms is given by standard matrix multiplication, which is well-defined because each column contains only one group element.

It is straightforward to see that there is an isomorphism of categories $G \wr \SE_{n} \simeq \mathcal{D}_{n}$. To each morphism $(f, \alpha) \in (G \wr \SE_{n})([k], [r])$, we associate an $r \times k$ matrix $[f, \alpha]$ defined by:
\[ [f, \alpha]_{i,j} = \begin{cases} 
f(j) & \text{if } \alpha(j) = i, \\
0 & \text{otherwise.} 
\end{cases} \]
By construction, the $j$-th column of $[f, \alpha]$ contains a unique non-zero element in the $\alpha(j)$-th row. The condition that $\alpha$ is surjective implies that every row contains at least one non-zero element. It is routine to verify that this assignment respects composition and thus defines an isomorphism of categories.

\subsection{The quiver of a skeletal EI-category algebra}

A category is called an \textit{EI-category} if its endomorphism monoids are groups. In other words, in an EI-category, every endomorphism is an isomorphism.

The category $G \wr \SE_{n}$ is an EI-category because for any morphism $(f,\alpha) \colon [k] \to [k]$, the map $\alpha$ is a surjective map from a finite set to itself, which is necessarily a bijection. Therefore, the endomorphism monoid of an object $[k]$ is the group $G \wr S_k$.

The problem of finding the ordinary quiver of the complex algebra of a skeletal EI-category can be reduced to a problem in group representation theory. We define a few concepts below and then state the relevant theorem.

For any finite set $X$, we denote by $\mathbb{C}X$ (or $\mathbb{C}[X]$) the complex vector space consisting of all formal linear combinations of elements of $X$. If a group $G$ acts on $X$, then $\mathbb{C}X$ naturally becomes a $\mathbb{C}G$-module called a \emph{permutation module}.

Let $\mathcal{A}$ be a finite EI-category. A morphism $m \in \mathcal{A}^1$ is called \textit{irreducible} if it is not an isomorphism, and whenever $m = m_1 m_2$, either $m_1$ or $m_2$ is an isomorphism. We denote the set of irreducible morphisms from object $c$ to object $c'$ by $\Irr(\mathcal{A})(c, c')$. Recall that we denote the set of simple modules of a group $G$ by $\IRep(G)$.
\begin{thm}[{\cite[Theorem 6.13]{Margolis2012}, \cite[Theorem 4.7]{Li2011}}] \label{thm:QuiverOfEICategories}
Let $\mathcal{A}$ be a finite skeletal EI-category and let $Q$ be the quiver of $\mathbb{C}\mathcal{A}$. Then:
\begin{enumerate}
    \item The set of vertices of $Q$ is given by 
    \[
    \bigsqcup_{c \in \mathcal{A}^0} \IRep(\mathcal{A}(c,c)).
    \]
    \item The vector space $\mathbb{C}[\Irr(\mathcal{A})(c, c')]$ can be viewed as an $(\mathcal{A}(c',c') \times \mathcal{A}(c,c))$-module with the action given by $(h, g) \bullet f = h f g^{-1}$. For $V \in \IRep(\mathcal{A}(c,c))$ and $U \in \IRep(\mathcal{A}(c',c'))$, the number of arrows from $V$ to $U$ is the multiplicity of $U \otimes V^*$ as a simple constituent in $\mathbb{C}[\Irr(\mathcal{A})(c, c')]$.
\end{enumerate}
\end{thm}

In view of the theorem above, the vertices of the quiver $Q$ of $\mathbb{C}(G \wr \SE_n)$ are indexed by the simple modules of the automorphism groups $G \wr S_k$ for each object $[k]$ in $G \wr \SE_n$. Since $G$ has $l$ conjugacy classes, these modules are indexed by multipartitions (or equivalently, multi-Young diagrams) $\Lambda$ with $k$ boxes and $l$ components, where $k$ varies from $0$ to $n$.

\subsection{Irreducible morphisms}
The next step for using \Cref{thm:QuiverOfEICategories} is identifying the irreducible morphisms of $G \wr \SE_{n}$.

\begin{lem} \label[lemma]{lem:IrreducibleMorphisms}
The irreducible morphisms of $G \wr \SE_{n}$ are precisely the morphisms from $[k+1]$ to $[k]$ for $0 \leq k < n$. In other words,
\[
\Irr(G \wr \SE_{n})([p],[k]) = 
\begin{cases} 
G \wr \SE_{n}([p],[k]) & \text{if } p = k+1, \\
\emptyset & \text{otherwise.}
\end{cases}
\]
\end{lem}

\begin{proof}
It is clear that every morphism $(f,\alpha)$ from $[k+1]$ to $[k]$ is irreducible. Indeed, if one decomposes $(f,\alpha) = (f_{1},\alpha_{1}) \cdot (f_{2},\alpha_{2})$, the size of the sets implies that one of the factors must be an endomorphism and hence an isomorphism. 

Now, if $(f,\alpha) \in G \wr \SE_{n}([p],[k])$ where $p > k+1$, then it is known that we can write $\alpha = \alpha_{1}\alpha_{2}$ for some onto functions $\alpha_{2} \colon [p] \to [k+1]$ and $\alpha_{1} \colon [k+1] \to [k]$ (see \cite[Lemma 3.3]{Stein2016}). In this case, $(\mathbf{1}_{[k+1]},\alpha_{1})$ and $(f,\alpha_{2})$ are both well-defined morphisms which are not endomorphisms and therefore are not invertible. Finally, 
\[
(\mathbf{1}_{[k+1]},\alpha_{1}) \cdot (f,\alpha_{2}) = (\mathbf{1}_{[k+1]}\ast\alpha_{2} \cdot f, \alpha_{1}\alpha_{2}) = (\mathbf{1}_{[p]} \cdot f, \alpha_{1}\alpha_{2}) = (f, \alpha),
\]
so $(f,\alpha)$ is not irreducible.
\end{proof}

Let $V \in \IRep(G \wr S_{p})$ and $U \in \IRep(G \wr S_{k})$. If $p \neq k+1$, there are no arrows in the quiver of $\mathbb{C}(G \wr \SE_{n})$ from $V$ to $U$ because $\Irr(G \wr \SE_{n})([p],[k])$ is empty. Consequently, we focus on the case $p = k+1$ and examine the module $\mathbb{C}[\Irr(G \wr \SE_{n})([k+1], [k])]$ under the action of $G \wr S_{k} \times G \wr S_{k+1}$ as described in \Cref{thm:QuiverOfEICategories}.

For convenience, we denote the set of irreducible morphisms by
\[
X = \Irr(G \wr \SE_{n})([k+1], [k]) = \{ (f, \alpha) \mid \alpha \colon [k+1] \to [k] \text{ is onto}, \, f \colon [k+1] \to G \}.
\]
The group $G \wr S_{k} \times G \wr S_{k+1}$ acts on $X$ via
\[
((f_{1}, \sigma_{1}), (f_{2}, \sigma_{2})) \bullet (f, \alpha) = (f_{1}, \sigma_{1}) \cdot (f, \alpha) \cdot (f_{2}, \sigma_{2})^{-1},
\]
where $f_{1} \colon [k] \to G$, $f_{2} \colon [k+1] \to G$, $\sigma_{1} \in S_{k}$, and $\sigma_{2} \in S_{k+1}$. The module of interest is the linearization $\mathbb{C}X$, which is a permutation module for this action.

\subsection{Description of the action and stabilizer}

\begin{lem} \label[lemma]{lem:TransitiveAction}
The action of $G \wr S_{k} \times G \wr S_{k+1}$ on $X$ is transitive.
\end{lem}

\begin{proof}
We show that every element of $X$ lies in the orbit of $(\mathbf{1}_{[k+1]}, \dec)$, where $\dec \colon [k+1] \to [k]$ is defined by 
\[
\dec(i) = 
\begin{cases} 
i & \text{if } i \leq k, \\
k & \text{if } i = k+1.
\end{cases}
\]

Let $(f, \alpha) \in X$. It is easy to verify (also mentioned in \cite[Section 3]{Stein2016}) that the action of $S_{k} \times S_{k+1}$ on the set of surjective maps from $[k+1]$ to $[k]$ is transitive. Thus, there exist $\sigma_{1} \in S_{k}$ and $\sigma_{2} \in S_{k+1}$ such that $\sigma_{1} \dec \sigma_{2} = \alpha$. 

Taking $(\mathbf{1}_{[k]}, \sigma_{1}) \in G \wr S_{k}$ and $(f, \sigma_{2})^{-1} \in G \wr S_{k+1}$, we observe that:
\begin{align*}
(\mathbf{1}_{[k]}, \sigma_{1}) \cdot (\mathbf{1}_{[k+1]}, \dec) \cdot \left((f, \sigma_{2})^{-1}\right)^{-1} 
    &= (\mathbf{1}_{[k]} \ast\dec \cdot \mathbf{1}_{[k+1]}, \sigma_{1} \dec) \cdot (f, \sigma_{2}) \\
    &= (\mathbf{1}_{[k+1]}, \sigma_{1} \dec) \cdot (f, \sigma_{2}) \\
    &= (\mathbf{1}_{[k+1]} \ast\sigma_{2} \cdot f, \sigma_{1} \dec \sigma_{2}) \\
    &= (\mathbf{1}_{[k+1]} \cdot f, \sigma_{1} \dec \sigma_{2}) \\
    &= (f, \alpha).
\end{align*}
Thus, the action is indeed transitive.
\end{proof}

Let $K$ be the stabilizer of $(\mathbf{1}_{[k+1]}, \dec)$ under the action of $G \wr S_{k} \times G \wr S_{k+1}$ described above, and let $\tr_{K}$ denote the trivial module of $K$. Since $\mathbb{C}X$ is a permutation module arising from a transitive group action, it follows that 
\[
\mathbb{C}X \simeq \Ind_{K}^{G \wr S_{k} \times G \wr S_{k+1}} (\tr_{K}).
\]
To proceed, we must characterize the stabilizer $K$ more explicitly.

For the action of $S_{k} \times S_{k+1}$ on the set of surjective maps from $[k+1]$ to $[k]$, it is known that the stabilizer of $\dec$ is the subgroup
\[
\{(\sigma, \sigma\tau) \mid \sigma \in S_{k-1}, \, \tau \in S_{\{k, k+1\}}\}
\]
(see \cite[Lemma 3.5]{Stein2016}). For any function $f \in G^{[k]}$, let $\hat{f} \in G^{[k+1]}$ be the function defined by
\[
\hat{f}(i) = 
\begin{cases} 
f(i) & \text{if } i \leq k, \\
f(k) & \text{if } i = k+1.
\end{cases}
\]
Note, in particular, that $\hat{f}(k+1) = \hat{f}(k)$.

Recall that we can view any element $\sigma \in S_{k-1}$ as an element of $S_k$ that fixes $k$.

\begin{lem} \label[lemma]{lem:StabilizerDescription}
The stabilizer $K$ of $(\mathbf{1}_{[k+1]}, \dec)$ under the action of $G \wr S_{k} \times G \wr S_{k+1}$ is given explicitly by
\[
K = \{ ((f, \sigma), (\hat{f}, \sigma\tau)) \mid \sigma \in S_{k-1}, \, \tau \in S_{\{k, k+1\}}, \, f \in G^{[k]} \}.
\]
\end{lem}

\begin{proof}
An element $((f, \sigma), (h, \epsilon))$ is in $K$ if and only if 
\[
(f, \sigma) \cdot (\mathbf{1}_{[k+1]}, \dec) \cdot (h, \epsilon)^{-1} = (\mathbf{1}_{[k+1]}, \dec),
\]
or equivalently,
\[
(f, \sigma) \cdot (\mathbf{1}_{[k+1]}, \dec) = (\mathbf{1}_{[k+1]}, \dec) \cdot (h, \epsilon).
\]
Applying the product rule for the wreath product on both sides, we obtain
\[
(f \ast \dec \cdot \mathbf{1}_{[k+1]}, \sigma \dec) = (\mathbf{1}_{[k+1]} \ast \epsilon \cdot h, \dec \epsilon).
\]
Since $\mathbf{1}_{[k+1]}$ is the identity for the pointwise product, this simplifies to the condition
\[
(f \ast \dec, \sigma \dec) = (h, \dec \epsilon).
\]
The map equality $\sigma \dec = \dec \epsilon$ implies that $(\sigma, \epsilon)$ is in the stabilizer of $\dec$ under the action of $S_{k} \times S_{k+1}$. It follows that $\sigma \in S_{k-1}$ and $\epsilon = \sigma \tau$ for some $\tau \in S_{\{k, k+1\}}$.

Finally, the function equality $f \ast \dec = h$ is equivalent to $h = \hat{f}$. Indeed, for $i \leq k$, we have $h(i) = f(\dec(i)) = f(i)$, and for $i = k+1$, we have $h(k+1) = f(\dec(k+1)) = f(k)$. This completes the proof.
\end{proof}

For any $f \in G^{[k-1]}$ and $g \in G$, we define a function $f^g \in G^{[k]}$ by 
\[
f^g(i) = \begin{cases} 
f(i) & i \leq k-1 \\ 
g & i = k
\end{cases}.
\]
Note that every $h \in G^{[k]}$ can be uniquely written as $h = f^g$ for some $f \in G^{[k-1]}$ and $g \in G$. 

In what follows, it will be useful to understand how elements of $K$ correspond to matrices over $G\cup \{0\}$. Given an element $(f^g, \sigma)$ with $\sigma \in S_{k-1}$, its associated matrix consists of an $(k-1) \times (k-1)$ block $A$ associated with $(f, \sigma)$ and a $1 \times 1$ block containing $g$. Thus, $(f^g, \sigma)$ corresponds to a block-diagonal matrix:
\[
A \oplus (g) = \begin{pmatrix}
\text{\Large $A$} & 0 \\
0 & g
\end{pmatrix}.
\]
For the element $(\widehat{f^g}, \sigma\tau)$, the associated matrix similarly decomposes into two blocks of sizes $(k-1) \times (k-1)$ and $2 \times 2$. The first block is $A$, and the second is the $2 \times 2$ block $g P_\tau$, where $P_\tau$ is the permutation matrix associated with $\tau$:
\[
\begin{pmatrix}
\text{\Large $A$} & \mathbf{0} \\
\mathbf{0} & g P_\tau
\end{pmatrix}.
\]

\begin{lem}
There is an isomorphism of groups
\[
K \simeq (G \wr S_{k-1}) \times G \times S_{2}.
\]
\end{lem}

\begin{proof}
Identify $S_{2}$ with $S_{\{k, k+1\}}$. We define a map $\psi: (G \wr S_{k-1}) \times G \times S_{2} \to K$ by 
\[
\psi((f, \sigma), g, \tau) = \left( (f^g, \sigma), (\widehat{f^g}, \sigma\tau) \right).
\]
By \Cref{lem:StabilizerDescription}, it is clear that $\psi$ is a bijection. It remains to show that $\psi$ is a group homomorphism.
Indeed,

\begin{align*}
\psi(((f_{1},\sigma_{1}),g_{1},\tau_{1})\cdot((f_{2},\sigma_{2}),g_{2},\tau_{2})) & =\psi(((f_{1},\sigma_{1})\cdot(f_{2},\sigma_{2})),g_{1}g_{2},\tau_{1}\tau_{2})\\
 & =\psi((f_{1}\ast\sigma_{2}\cdot f_{2},\sigma_{1}\sigma_{2}),g_{1}g_{2},\tau_{1}\tau_{2})\\
 & =(((f_{1}\ast\sigma_{2}\cdot f_{2})^{g_{1}g_{2}},\sigma_{1}\sigma_{2}),(\widehat{(f_{1}\ast\sigma_{2}\cdot f_{2})^{g_{1}g_{2}}},\sigma_{1}\sigma_{2}\tau_{1}\tau_{2}))
\end{align*}
On the other hand
\begin{align*}
\psi((f_{1},\sigma_{1}),g_{1},\tau_{1})\cdot\psi((f_{2},\sigma_{2}),g_{2},\tau_{2}) & =((f_{1}^{g_{1}},\sigma_{1}),(\widehat{f_{1}^{g_{1}}},\sigma_{1}\tau_{1}))\cdot((f_{2}^{g_{2}},\sigma_{2}),(\widehat{f_{2}^{g_{2}}},\sigma_{2}\tau_{2}))\\
 & =((f_{1}^{g_{1}}\ast\sigma_{2}\cdot f_{2}^{g_{2}},\sigma_{1}\sigma_{2}),(\widehat{f_{1}^{g_{1}}}\ast\sigma_{2}\tau_{2}\cdot\widehat{f_{2}^{g_{2}}},\sigma_{1}\tau_{1}\sigma_{2}\tau_{2})).
\end{align*}

First note that 
\[
\sigma_{1}\sigma_{2}\tau_{1}\tau_{2} = \sigma_{1}\tau_{1}\sigma_{2}\tau_{2}
\]
because $\tau_{1}$ and $\sigma_{2}$ have disjoint supports and thus commute as elements of $S_{k+1}$.
Next, for $i \leq k-1$, we have
\[
(f_{1}\ast\sigma_{2}\cdot f_{2})^{g_{1}g_{2}}(i) = f_{1}(\sigma_{2}(i)) \cdot f_{2}(i) = f_{1}^{g_{1}}(\sigma_{2}(i)) \cdot f_{2}^{g_{2}}(i)
\]
because $\sigma_{2}(i) \leq k-1$ as well. For $i=k$, we have 
\[
(f_{1}\ast\sigma_{2}\cdot f_{2})^{g_{1}g_{2}}(k) = g_{1}g_{2} = f_{1}^{g_{1}}(k) \cdot f_{2}^{g_{2}}(k) = f_{1}^{g_{1}}(\sigma_{2}(k)) \cdot f_{2}^{g_{2}}(k)
\]
since $\sigma_2$ fixes $k$. This establishes that
\[
(f_{1}\ast\sigma_{2}\cdot f_{2})^{g_{1}g_{2}} = f_{1}^{g_{1}}\ast\sigma_{2} \cdot f_{2}^{g_{2}}.
\]
Proving the last equality is similar. For $i \leq k-1$, we have again 
\[
\widehat{(f_{1}\ast\sigma_{2}\cdot f_{2})^{g_{1}g_{2}}}(i) = f_{1}(\sigma_{2}(i)) \cdot f_{2}(i)
\]
and 
\begin{align*}
\left(\widehat{f_{1}^{g_{1}}}\ast\sigma_{2}\tau_{2} \cdot \widehat{f_{2}^{g_{2}}}\right)(i) 
& = \widehat{f_{1}^{g_{1}}}(\sigma_{2}\tau_{2}(i)) \cdot \widehat{f_{2}^{g_{2}}}(i) \\
& = f_{1}(\sigma_{2}(i)) \cdot f_{2}(i)
\end{align*}
because $\tau_{2}(i) = i$ for all $i \leq k-1$.
Finally, for $k \leq i \leq k+1$, we have
\[
\widehat{(f_{1}\ast\sigma_{2} \cdot f_{2})^{g_{1}g_{2}}}(i) = (f_{1}\ast\sigma_{2} \cdot f_{2})^{g_{1}g_{2}}(k) = g_{1}g_{2}
\]
and
\[
\left(\widehat{f_{1}^{g_{1}}}\ast\sigma_{2}\tau_{2} \cdot \widehat{f_{2}^{g_{2}}}\right)(i) = \widehat{f_{1}^{g_{1}}}(\sigma_{2}\tau_{2}(i)) \cdot \widehat{f_{2}^{g_{2}}}(i) = g_{1}g_{2},
\]
where the last equality follows because $\sigma_2$ fixes $\{k, k+1\}$ and $\widehat{f_j^{g_j}}$ takes the value $g_j$ on both $k$ and $k+1$. This completes the proof.
\end{proof}

\subsection{The quiver computation}
Let $U \in \IRep G \wr S_{k}$ and $V \in \IRep G \wr S_{k+1}$. To simplify the notation, we shall use the same notation for a module
and its character. For example, we write $U$ and $V$ in place of $\chi_{U}$ and $\chi_{V}$.

The number of arrows in the quiver of $\mathbb{C}(G \wr \SE_{n})$ from $V$ to $U$ is the multiplicity of $U \otimes V^{\ast}$ as a simple module in the $(G \wr S_{k} \times G \wr S_{k+1})$-module 
\[
\mathbb{C}X = \Ind_{K}^{G \wr S_{k} \times G \wr S_{k+1}} \tr_{K}.
\]
This multiplicity can be expressed as the inner product of characters:
\[
\langle U \otimes V^{\ast}, \Ind_{K}^{G \wr S_{k} \times G \wr S_{k+1}} \tr_{K} \rangle.
\]
Using Frobenius reciprocity, this equals 
\begin{align*}
\langle\Res_{K}^{G\wr S_{k}\times G\wr S_{k+1}}(U\otimes V^{\ast}),\tr_{K}\rangle. 
\end{align*}
According to \Cref{lem:StabilizerDescription}, a general element of $K$ is of the form \[((f^g, \sigma), (\hat{f^g}, \sigma\tau))\]
where $\sigma \in S_{k-1}, \, \tau \in S_{\{k, k+1\}}, \, f \in G^{[k-1]},\, g\in G$. Therefore,  
\begin{align*}
\MoveEqLeft \langle\Res_{K}^{G\wr S_{k}\times G\wr S_{k+1}}(U\otimes V^{\ast}), \tr_{K}\rangle \\
&= \frac{1}{|K|}\sum_{\substack{f\in G^{[k-1]}, g\in G \\ \sigma \in S_{k-1}, \tau \in S_2}} U\otimes V^\ast (((f^g, \sigma), (\hat{f^g}, \sigma\tau))) \tr_K(((f^g, \sigma), (\hat{f^g}, \sigma\tau))) \\
&= \frac{1}{|K|}\sum_{\substack{f\in G^{[k-1]}, g\in G \\ \sigma \in S_{k-1}, \tau \in S_2}} U((f^g, \sigma))\cdot V^\ast ((\hat{f^g}, \sigma\tau)) \\
&= \frac{1}{|K|}\sum_{\substack{f\in G^{[k-1]}, g\in G \\ \sigma \in S_{k-1}, \tau \in S_2}} U((f^g, \sigma))\cdot \overline{V ((\hat{f^g}, \sigma\tau))}.
\end{align*}
Now, if we think of $(\hat{f^g}, \sigma\tau)$ as a $(k-1)\times (k-1)$ and $2\times 2$ block matrix, it is just a general element of $G\wr S_{k-1}\times G\times S_2$ so we can write
\[
V ((\hat{f^g}, \sigma\tau))=\Res^{G\wr S_{k+1}}_{G\wr S_{k-1} \times (G\times S_2)}V ((\hat{f^g}, \sigma\tau))
\]
Likewise, if we view $(f^g, \sigma)$ as a $(k-1)\times (k-1)$ and $1\times 1$ block matrix, it is a general element of $G\wr S_{k-1}\times G$ so
\[
U ((f^g, \sigma))=\Res^{G\wr S_{k}}_{G\wr S_{k-1} \times G}U ((f^g, \sigma)).
\]
In order to view this also as a $K=(G\wr S_{k-1}\times(G\times S_2))$ module we will write this as
\[
\Res^{G\wr S_{k}}_{G\wr S_{k-1} \times G}U ((f^g, \sigma))\cdot \tr_2(\tau).
\]
where $\tr_2$ is the trivial module of $S_2$. Therefore, we obtain
\begin{align*}
\MoveEqLeft \langle\Res_{K}^{G\wr S_{k}\times G\wr S_{k+1}}(U\otimes V^{\ast}), \tr_{K}\rangle \\
&= \frac{1}{|K|}\sum_{\substack{f\in G^{[k-1]}, g\in G \\ \sigma \in S_{k-1}, \tau \in S_2}} U((f^g, \sigma)) \cdot \overline{V ((\hat{f^g}, \sigma\tau))} \\
&= \frac{1}{|K|}\sum_{\substack{f\in G^{[k-1]}, g\in G \\ \sigma \in S_{k-1}, \tau \in S_2}} \left(\Res^{G\wr S_{k}}_{G\wr S_{k-1} \times G}U ((f^g, \sigma)) \cdot \tr_2(\tau)\right) \cdot \overline{\Res^{G\wr S_{k+1}}_{G\wr S_{k-1} \times (G\times S_2)}V ((\hat{f^g}, \sigma\tau))} \\
&= \frac{1}{|K|}\sum_{\substack{((f,\sigma),(g,\tau)) \\ \in (G\wr S_{k-1})\times(G\times S_2)}} ((\Res^{G\wr S_{k}}_{G\wr S_{k-1} \times G}U) \boxtimes \tr_2)((f^g, \sigma, \tau)) \cdot \overline{\Res^{G\wr S_{k+1}}_{G\wr S_{k-1} \times (G\times S_2)}V ((\hat{f^g}, \sigma\tau))} \\
&= \langle (\Res^{G\wr S_{k}}_{G\wr S_{k-1} \times G}U) \boxtimes \tr_2, \Res^{G\wr S_{k+1}}_{G\wr S_{k-1} \times (G\times S_2)}V \rangle.
\end{align*}
By Frobenius reciprocity, this equals 
\[\langle \Ind^{G\wr S_{k+1}}_{G\wr S_{k-1} \times (G\times S_2)}((\Res^{G\wr S_{k}}_{G\wr S_{k-1} \times G}U) \boxtimes \tr_2), V \rangle.\]
By transitivity of induction we obtain
\[\langle \Ind^{G\wr S_{k+1}}_{G\wr S_{k-1} \times (G \wr S_2)}\Ind^{G\wr S_{k-1} \times (G \wr S_2)}_{G\wr S_{k-1} \times (G\times S_2)}((\Res^{G\wr S_{k}}_{G\wr S_{k-1} \times G}U) \boxtimes \tr_2), V \rangle.\]

Now we analyze this expression using the branching rules described in \Cref{sec:branching_rules}. Let $U$ be a simple module of $G \wr S_k$ corresponding to a multipartition $\Lambda$. 

By \Cref{prop:Simple_Branching}, we have 
\[ 
\Res^{G \wr S_{k+1}}_{(G \wr S_k) \times G} (S^\Lambda) = \bigoplus_{r=1}^{l} \bigoplus_{\Gamma \in Y^-_r(\Lambda)} S^\Gamma \boxtimes U_r, 
\]
where $\Gamma \in Y^-_r(\Lambda)$ means that $\Gamma$ is obtained by removing one box from the $r$-th component of $\Lambda$.
Let $S^\Gamma \boxtimes U_r$ be one component in this summation. The next step is to compute the induction 
\[\Ind^{G\wr S_{k-1} \times (G \wr S_2)}_{G\wr S_{k-1} \times (G\times S_2)}(S^\Gamma \boxtimes U_r \boxtimes \tr_2)\]
For this we need to compute the induction 
\[\Ind^{G \wr S_2}_{ (G\times S_2)}(U_r \boxtimes \tr_2).\]
Let $m^r_{i,j}$ be the multiplicity of $U_r$ in $U_i\otimes U_j$, let $m^{+,r}_i$ be the multiplicity of $U_r$ in $\Sym^2U_i$, and let $m^{-,r}_i$ be the multiplicity of $U_r$ in $\Alt^2U_i$.
According to \Cref{lem:Induction_GwrS_2Wij} and \Cref{lem:Induction_GwrS_2} we have that
\[\Ind^{G \wr S_2}_{ (G\times S_2)}(U_r \boxtimes \tr_2)=\bigoplus_{1\leq i<j\leq l} \left(m^r_{i,j}W_{i,j}\right)\oplus \bigoplus_{i=1}^l \left(m^{+,r}_iW^+_i\oplus m^{-,r}_iW^-_i\right)\]
so
\[\Ind^{G\wr S_{k-1} \times (G \wr S_2)}_{G\wr S_{k-1} \times (G\times S_2)}\left(S^\Gamma \boxtimes U_r \boxtimes \tr_2\right)=\bigoplus_{1\leq i<j\leq l}\left(m^r_{i,j}(S^\Gamma \boxtimes W_{i,j})\right)\oplus \bigoplus _{i=1}^l \left(m^{+,r}_i(S^\Gamma \boxtimes W^+_i)\oplus m^{-,r}_i(S^\Gamma \boxtimes W^-_i)\right).\]
Finally, by \Cref{Inductopn_two_steps} we have that
\[\Ind^{G\wr S_{k+1}}_{G\wr S_{k-1} \times (G \wr S_2)}\left(\bigoplus_{1\leq i<j\leq l}m^r_{i,j}(S^\Gamma \boxtimes W_{i,j})\oplus \bigoplus _{i=1}^l \left(m^{+,r}_i(S^\Gamma \boxtimes W^+_i)\oplus m^{-,r}_i(S^\Gamma \boxtimes W^-_i)\right)\right)=\]
\[\bigoplus_{1\leq i<j\leq l }\left(m^r_{i,j} \bigoplus_{\Delta \in Y^+_{i,j}(\Gamma)} S^{\Delta}\right)\oplus \bigoplus_{i=1}^l \left( m_i^{+,r}\bigoplus_{\Delta \in Y^+_{i,H^2}(\Gamma)} S^{\Delta}\oplus m_i^{-,r}\bigoplus_{\Delta \in Y^+_{i,V^2}(\Gamma)} S^{\Delta}.\right)\]
From this we can conclude:
\begin{thm} \label{thm:Quiver_of_GwrE_n}
Let $G$ be a group with $l$ conjugacy classes and let $\IRep G=\{U_1,\ldots,U_l\}$. Define $m^r_{i,j},m^{+,r}_i,m^{-,r}_i$ as above. The quiver $Q$ of the monoid $G\wr \PT_n$ is described as follows. The vertices correspond to multi-Young diagrams with $l$ components and $k$ boxes, where $k$ varies from $0$ to $n$. Let $U$ be a multi-Young diagram with $k$ boxes, and $V$ be a multi-Young diagram with $p$ boxes. If $p\neq k+1$, then there are no arrows from $V$ to $U$. If $p=k+1$, then the number of arrows depends on the ways that $V$ can be constructed from $U$ by removing one box and adding two. For each way that $V$ can be constructed from $U$ by removing one box from the $r$-th component and adding one box in components $i$ and $j$ ($i\neq j$) we add $m^r_{i,j}$ arrows. In each way that $V$ can be constructed from $U$ by removing one box from the $r$-th component and adding two boxes on the $i$-th component but not on the same column we add $m^{+,r}_i$ arrows. In each way that $V$ can be constructed from $U$ by removing one box from the $r$-th component and adding two boxes on the $i$-th component but not on the same row we add $m^{-,r}_i$ arrows.
\end{thm}
\subsection{Examples}
\subsubsection{The quiver of $\mathbb{C}\PT_n$}
If $G$ is the trivial group then $G\wr \PT_n\simeq \PT_n$. In this case, $\IRep G$ contains only the trivial module $U_1=\tr_G$. Clearly $\Sym^2 \tr_G=\tr_G$ and $\Alt^2 \tr_G=0$. Therefore, $m^{+,1}_1=1$ and $m^{-,1}_1=0$. Therefore, we retrieve \cite[Theorem 3.8]{Stein2016}:
\begin{thm}
The vertices of the quiver $Q$ of the monoid algebra $\mathbb{C}\PT_n$ correspond to Young diagrams with $k$ boxes where $k$ varies from $0$ to $n$. Let $U$ be a Young diagram with $k$ boxes and $V$ be a Young diagram with $p$ boxes. If $p\neq k+1$, then there are no arrows from $V$ to $U$. If $p=k+1$, then the number of arrows from  $V$ to $U$ is the number of ways that $V$ can be constructed from $U$ by removing one box and adding two, but not in the same column.
\end{thm}
\subsubsection{Generalized monoid of partial functions}
Let $G=C_l$ be the cyclic group with $l$ elements, say $G = \{1_G, g, g^2, \dots, g^{l-1}\}$ where $g^l = 1_G$. 
All simple modules $\rho \in \text{IRep } G$ are one-dimensional. There are $l$ such modules, $\rho_r$ for $r = 0, \dots, l-1$, defined by:
$$\rho_r(g^n) = \omega^{rn}$$
where $\omega = e^{2\pi i / l}$ is a primitive $l$-th root of unity. These modules satisfy $\rho_i \otimes \rho_j = \rho_{i+j}$, where the sum is taken modulo $l$. Since each $\rho_r$ is one-dimensional, its symmetric square and alternating square satisfy:
$$\text{Sym}^2 \rho_r = \rho_r \otimes \rho_r = \rho_{2r} \quad \text{and} \quad \text{Alt}^2 \rho_r = 0$$
Therefore, the decomposition coefficients $m_{i,j}^r$ for the tensor product are given by $m_{i,j}^r = 1$ if $i+j \equiv r \pmod{l}$ and zero otherwise. For the symmetric and alternating squares, the coefficients are $m^{+,r}_i = 1$ if $2i \equiv r \pmod{l}$ and zero otherwise, while $m^{-,r}_i = 0$ for all $r \in \{0, \dots, l-1\}$. We end with the following result:
\begin{thm}
The vertices of the quiver $Q$ of the monoid algebra $\mathbb{C}(C_l\wr\PT_n)$ correspond to mutli-Young diagrams with $l$ components, indexed $\{0,\ldots l-1\}$, and $k$ boxes, where $k$ varies from $0$ to $n$. Let $U$ be a multi-Young diagram with $k$ boxes and $V$ be a multi-Young diagram with $p$ boxes. If $p\neq k+1$, then there are no arrows from $V$ to $U$. If $p=k+1$, then the number of arrows depends on the ways that $V$ can be constructed from $U$ by removing one box and adding two. For each way that $V$ can be constructed from $U$ by removing one box from the $r$-th component and adding one box in components $i$and $j$ we add an arrow if $i+j \equiv r \pmod{l}$. In each way that $V$ can be constructed from $U$ by removing one box from the $r$-th component and adding two boxes on the $i$-th component but not on the same column we add an arrow if $2i \equiv r \pmod{l}$.
\end{thm}
\begin{exmp}
Consider the case $G=C_2$. The quiver of the algebra $\mathbb{C}(C_2 \wr \PT_3)$ is given by the following figure:

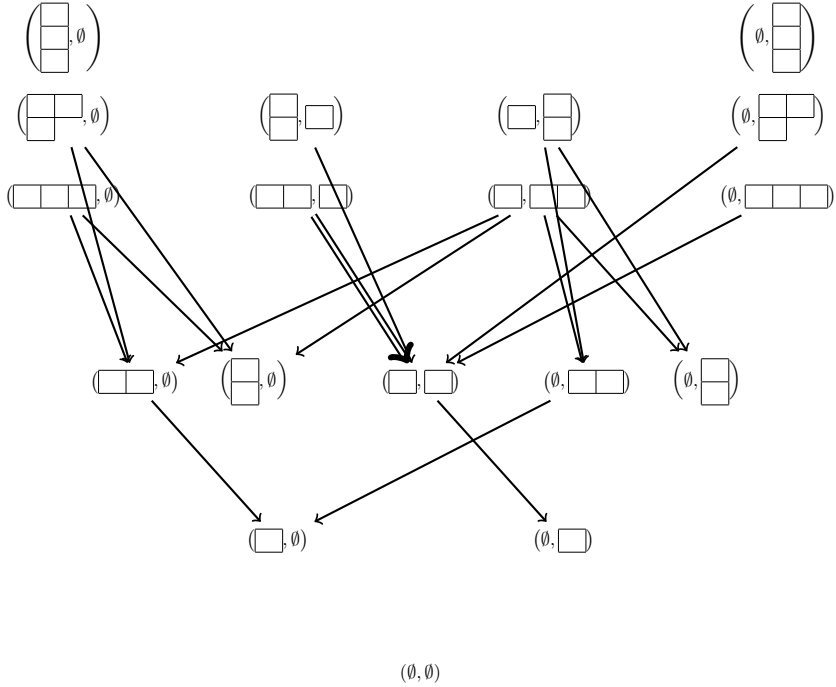
\begin{tikzpicture}[node distance=2cm, auto, thick, scale=0.7, xscale=0.9, every node/.style={scale=0.7}]

  \node (v0) at (0, 0) {$\left( \emptyset , \emptyset \right)$};

  \node (v1_0) at (-3, 2.5) {$\left( \begin{tabular}{|c|} \hline \phantom{.} \\ \hline \end{tabular} , \emptyset \right)$};
  \node (v1_1) at (3, 2.5) {$\left( \emptyset , \begin{tabular}{|c|} \hline \phantom{.} \\ \hline \end{tabular} \right)$};

  \node (v2_0h) at (-6, 5.5) {$\left( \begin{tabular}{|c|c|} \hline \phantom{.} & \phantom{.} \\ \hline \end{tabular} , \emptyset \right)$};
  \node (v2_0v) at (-3.5, 5.5) {$\left( \begin{tabular}{|c|} \hline \phantom{.} \\ \hline \phantom{.} \\ \hline \end{tabular} , \emptyset \right)$};
  \node (v2_m)  at (0, 5.5)  {$\left( \begin{tabular}{|c|} \hline \phantom{.} \\ \hline \end{tabular} , \begin{tabular}{|c|} \hline \phantom{.} \\ \hline \end{tabular} \right)$};
  \node (v2_1h) at (3.5, 5.5)  {$\left( \emptyset , \begin{tabular}{|c|c|} \hline \phantom{.} & \phantom{.} \\ \hline \end{tabular} \right)$};
  \node (v2_1v) at (6, 5.5)  {$\left( \emptyset , \begin{tabular}{|c|} \hline \phantom{.} \\ \hline \phantom{.} \\ \hline \end{tabular} \right)$};

  \node (v3_30a) at (-7.5, 9) {$\left( \begin{tabular}{|c|c|c|} \hline \phantom{.} & \phantom{.} & \phantom{.} \\ \hline \end{tabular} , \emptyset \right)$};
  \node (v3_30b) at (-7.5, 10.5) {$\left( \begin{tabular}{|c|c|} \hline \phantom{.} & \phantom{.} \\ \hline \phantom{.} & \multicolumn{1}{c}{} \\ \cline {1-1} \end{tabular} , \emptyset \right)$};
  \node (v3_30c) at (-7.5, 12) {$\left( \begin{tabular}{|c|} \hline \phantom{.} \\ \hline \phantom{.} \\ \hline \phantom{.} \\ \hline \end{tabular} , \emptyset \right)$};
  
  \node (v3_21a) at (-2.5, 9) {$\left( \begin{tabular}{|c|c|} \hline \phantom{.} & \phantom{.} \\ \hline \end{tabular} , \begin{tabular}{|c|} \hline \phantom{.} \\ \hline \end{tabular} \right)$};
  \node (v3_21b) at (-2.5, 10.5) {$\left( \begin{tabular}{|c|} \hline \phantom{.} \\ \hline \phantom{.} \\ \hline \end{tabular} , \begin{tabular}{|c|} \hline \phantom{.} \\ \hline \end{tabular} \right)$};
  
  \node (v3_12a) at (2.5, 9) {$\left( \begin{tabular}{|c|} \hline \phantom{.} \\ \hline \end{tabular} , \begin{tabular}{|c|c|} \hline \phantom{.} & \phantom{.} \\ \hline \end{tabular} \right)$};
  \node (v3_12b) at (2.5, 10.5) {$\left( \begin{tabular}{|c|} \hline \phantom{.} \\ \hline \end{tabular} , \begin{tabular}{|c|} \hline \phantom{.} \\ \hline \phantom{.} \\ \hline \end{tabular} \right)$};
  
  \node (v3_03a) at (7.5, 9) {$\left( \emptyset , \begin{tabular}{|c|c|c|} \hline \phantom{.} & \phantom{.} & \phantom{.} \\ \hline \end{tabular} \right)$};
  \node (v3_03b) at (7.5, 10.5) {$\left( \emptyset , \begin{tabular}{|c|c|} \hline \phantom{.} & \phantom{.} \\ \hline \phantom{.} & \multicolumn{1}{c}{} \\ \cline {1-1} \end{tabular} \right)$};
  \node (v3_03c) at (7.5, 12) {$\left( \emptyset , \begin{tabular}{|c|} \hline \phantom{.} \\ \hline \phantom{.} \\ \hline \phantom{.} \\ \hline \end{tabular} \right)$};

  \draw[->] (v2_0h) -- (v1_0);
  \draw[->] (v2_m) -- (v1_1);
  \draw[->] (v2_1h) -- (v1_0);

  \draw[->] (v3_30a) -- (v2_0h);
  \draw[->] (v3_30b) -- (v2_0h);
  \draw[->] (v3_30a) -- (v2_0v);
  \draw[->] (v3_30b) -- (v2_0v);

  \draw[->, double, double distance=1.2pt] (v3_21a) -- (v2_m);
  \draw[->] (v3_21b) -- (v2_m);

  \draw[->] (v3_12a) -- (v2_1h);
  \draw[->] (v3_12b) -- (v2_1h);
  \draw[->] (v3_12a) -- (v2_1v);
  \draw[->] (v3_12b) -- (v2_1v);
  
  \draw[->] (v3_12a) -- (v2_0h);
  \draw[->] (v3_12a) -- (v2_0v);

  \draw[->] (v3_03a) -- (v2_m);
  \draw[->] (v3_03b) -- (v2_m);

\end{tikzpicture}

\noindent The first component of each multi-Young diagram corresponds to the trivial module $\rho_0$, while the second component corresponds to the alternating module $\rho_1$. 

\medskip

\noindent Notably, there are two arrows from 
$\scalebox{0.8}{$\left( \begin{tabular}{|c|c|} \hline \phantom{.} & \phantom{.} \\ \hline \end{tabular} \,,\, \begin{tabular}{|c|} \hline \phantom{.} \\ \hline \end{tabular} \right)$}$ 
to 
$\scalebox{0.8}{$\left( \begin{tabular}{|c|} \hline \phantom{.} \\ \hline \end{tabular} \,,\, \begin{tabular}{|c|} \hline \phantom{.} \\ \hline \end{tabular} \right)$}$. 
This multiplicity arises because the branching construction can be satisfied in two distinct ways: 
\begin{itemize}
    \item First, by removing one box from the first component ($r=0$) and adding two boxes in the same row, which is valid since $0+0 \equiv 0 \pmod{2}$.
    \item Second, by removing one box from the second component ($r=1$) and adding one box to each component, which is valid since $0+1 \equiv 1 \pmod{2}$.
\end{itemize}
\end{exmp}
\subsubsection{The quiver of $\mathbb{C}(S_3\wr \PT_n)$}
We now consider the case of the smallest non-abelian group $G = S_3$. The set $\IRep(G)$ contains three simple modules: The trivial module $U_1$, the sign module $U_2$, and the standard module $U_3$. It is well-known and easily verified by character considerations that the symmetric and alternating squares satisfy:
\begin{align*}
    \Sym^2 U_1 &= \Sym^2 U_2 = U_1, & \Sym^2 U_3 &= U_3 \oplus U_1, \\
    \Alt^2 U_1 &= \Alt^2 U_2 = 0,   & \Alt^2 U_3 &= U_2.
\end{align*}
Furthermore, the tensor products are given by:
\[
U_1 \otimes U_2 = U_2, \quad U_1 \otimes U_3 = U_3, \quad U_2 \otimes U_3 = U_3.
\]
Therefore, the coefficients $m_{i,j}^r$ are given by:
\[
m_{1,2}^r = \begin{cases} 
1 & \text{if } r = 2, \\
0 & \text{otherwise,}
\end{cases}
\quad \text{and} \quad
m_{1,3}^r = m_{2,3}^r = \begin{cases} 
1 & \text{if } r = 3, \\
0 & \text{otherwise.}
\end{cases}
\]
The coefficients $m_{1}^{q,r},m_{2}^{q,r}$, where $q\in\{+,-\}$,  satisfy:
\[
m_{1}^{q,r} = m_{2}^{q,r}= \begin{cases} 
1 & \text{if } q = + \text{ and } r = 1, \\
0 & \text{otherwise.}
\end{cases}
\]
Finally, 
for the coefficients $m_3^{q,r}$, where $q \in \{+, -\}$, we have:
\[
m_3^{+,r} = \begin{cases} 
1 & \text{if } r \in \{1, 3\}, \\
0 & \text{otherwise,}
\end{cases}
\quad \text{and} \quad
m_3^{-,r} = \begin{cases} 
1 & \text{if } r = 2, \\
0 & \text{otherwise.}
\end{cases}
\]
Therefore, we obtain:
\begin{thm}
The vertices of the quiver $Q$ of the monoid algebra $\mathbb{C}(S_3 \wr \PT_n)$ correspond to multi-Young diagrams $\boldsymbol{\Lambda} = (\lambda_{1}, \lambda_{2}, \lambda_{3})$ with $k$ total boxes, where $0 \leq k \leq n$. Let $U$ be a multi-Young diagram with $k$ boxes and $V$ be a multi-Young diagram with $p$ boxes. 

If $p \neq k+1$, there are no arrows from $V$ to $U$. If $p = k+1$, the number of arrows from $V$ to $U$ is the number of ways $V$ can be constructed from $U$ by removing one box and adding two boxes according to the following rules:
\begin{itemize}
    \item Remove one box from $\lambda_1$ and add two boxes to any one component, provided they are not in the same column.
    \item Remove one box from $\lambda_2$ and add one box to $\lambda_1$ and one box to $\lambda_2$.
    \item Remove one box from $\lambda_2$ and add two boxes to $\lambda_3$, provided they are not in the same row.
    \item Remove one box from $\lambda_3$ and add one box to $\lambda_1$ or $\lambda_2$ and another box to $\lambda_3$.
    \item Remove one box from $\lambda_3$ and add two boxes to $\lambda_3$, provided they are not in the same column.
\end{itemize}
\end{thm}

\begin{exmp}
The quiver of the algebra $\mathbb{C}(S_3 \wr \PT_2)$ is given by the following figure:

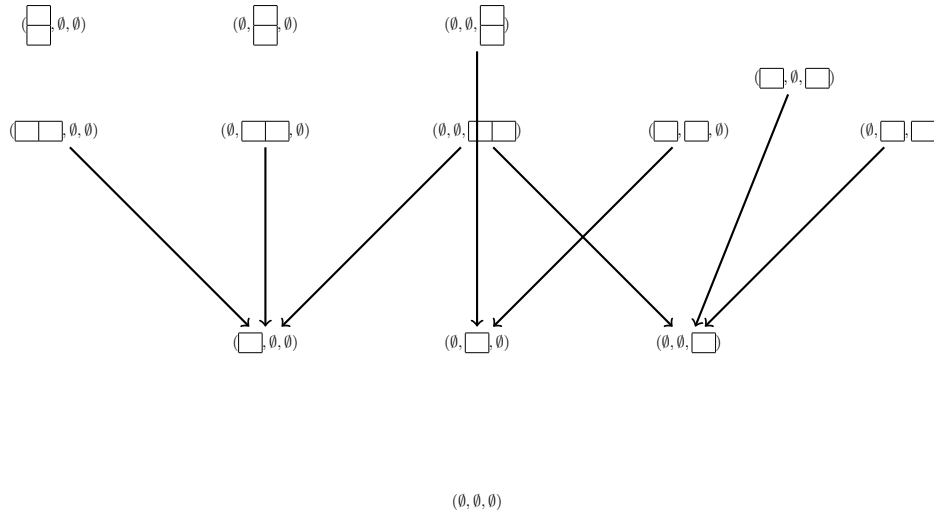
\begin{tikzpicture}[node distance=2.5cm, auto, thick, scale=0.7, every node/.style={scale=0.6}]

  \newcommand{\boxx}{\begin{tabular}{|c|} \hline \phantom{.} \\ \hline \end{tabular}}
  \newcommand{\hboxx}{\begin{tabular}{|c|c|} \hline \phantom{.} & \phantom{.} \\ \hline \end{tabular}}
  \newcommand{\vboxx}{\begin{tabular}{|c|} \hline \phantom{.} \\ \hline \phantom{.} \\ \hline \end{tabular}}

  \node (v0) at (0, 0) {$(\emptyset, \emptyset, \emptyset)$};

  \node (v1_1) at (-4, 3) {$(\boxx, \emptyset, \emptyset)$};
  \node (v1_2) at (0, 3)  {$(\emptyset, \boxx, \emptyset)$};
  \node (v1_3) at (4, 3)  {$(\emptyset, \emptyset, \boxx)$};

  \node (v2_1h) at (-8, 7) {$(\hboxx, \emptyset, \emptyset)$};
  \node (v2_1v) at (-8, 9) {$(\vboxx, \emptyset, \emptyset)$};
  
  \node (v2_2h) at (-4, 7) {$(\emptyset, \hboxx, \emptyset)$};
  \node (v2_2v) at (-4, 9) {$(\emptyset, \vboxx, \emptyset)$};
  
  \node (v2_3h) at (0, 7)  {$(\emptyset, \emptyset, \hboxx)$};
  \node (v2_3v) at (0, 9)  {$(\emptyset, \emptyset, \vboxx)$};

  \node (v2_12) at (4, 7)  {$(\boxx, \boxx, \emptyset)$};
  \node (v2_13) at (6, 8)  {$(\boxx, \emptyset, \boxx)$};
  \node (v2_23) at (8, 7)  {$(\emptyset, \boxx, \boxx)$};


  
  \draw[->] (v2_1h) -- (v1_1);
  \draw[->] (v2_2h) -- (v1_1);
  \draw[->] (v2_3h) -- (v1_1);

  \draw[->] (v2_12) -- (v1_2);
  \draw[->] (v2_3v) -- (v1_2);

  \draw[->] (v2_13) -- (v1_3);
  \draw[->] (v2_23) -- (v1_3);
  \draw[->] (v2_3h) -- (v1_3);

\end{tikzpicture}
\end{exmp}

\section{Global dimension of $\mathbb{C}(G\wr \PT_n)$}
\label{sec:global_dimension_of_PT_n}
In this section, we prove that the global dimension of the algebra $\mathbb{C}(G\wr \PT_n)$ is $n-1$ based on the results established for the algebra $\mathbb{C}\PT_n$ in \cite{Stein2019}. We show that every projective module of $\mathbb{C}\PT_n$ can be lifted to a corresponding projective module of $\mathbb{C}(G\wr \PT_n)$, allowing us to transfer the known homological properties to the wreath product case.

\subsection{The global dimension of an algebra}
Let $A$ be a finite-dimensional $\mathbb{C}$-algebra. An $A$-module $P$ is \emph{projective} if the functor $\text{Hom}_A(P, -)$ is exact, or equivalently, if $P$ is a direct summand of a free module. A module is \emph{indecomposable} if it cannot be written as a direct sum of two non-zero submodules. Every projective $A$-module decomposes into a direct sum of indecomposable projective modules.

These modules are classified by the idempotents of the algebra. Two idempotents $e, f$ are \emph{orthogonal} if $ef = fe = 0$. An idempotent is \emph{primitive} if it cannot be written as a sum of two non-zero orthogonal idempotents. A set $\{e_1, \dots, e_n\}$ is a \emph{complete set of orthogonal primitive idempotents} if the elements are pairwise orthogonal, each $e_i$ is primitive, and $\sum e_i = 1_A$. Every indecomposable projective $A$-module is isomorphic to $Ae$ for some primitive idempotent $e$, and a complete set of primitive orthogonal idempotents yields all indecomposable projectives up to isomorphism.

Let $A$ be a finite-dimensional $\mathbb{C}$-algebra. The \emph{Jacobson radical} of $A$, denoted $\Rad(A)$, is defined as the intersection of all maximal left ideals of $A$. Equivalently, $\Rad(A)$ is the unique maximal nilpotent two-sided ideal of $A$. For any finite-dimensional $A$-module $M$, we define its \emph{radical}, $\Rad(M)$, as the intersection of all maximal submodules of $M$. A fundamental property of the radical is that it can be computed via the action of the algebra's radical. Specifically, we have the identity $\Rad(M) = \Rad(A)M$.

Let $A$ be a finite-dimensional $\mathbb{C}$-algebra and $M$ an $A$-module. A \emph{projective cover} of $M$ is a pair $(P, \pi)$, where $P$ is a projective $A$-module and $\pi:P\to M$ is a surjection such that $\ker(\pi)\subseteq \Rad(P)$.

For the algebras considered here, every finite-dimensional module $M$ admits a projective cover, which is unique up to isomorphism.

Let $A$ be a finite-dimensional $\mathbb{C}$-algebra and $M$ an $A$-module. A \emph{projective resolution} of $M$ is an exact sequence of the form
\[ \mathbf{P}_\bullet: \dots \to P_n \xrightarrow{d_n} P_{n-1} \to \dots \to P_1 \xrightarrow{d_1} P_0 \xrightarrow{\epsilon} M \to 0 \]
where each $P_i$ is a projective $A$-module. Such a resolution is called \emph{minimal} if $P_0$ is the projective cover of $M$, and for each $i \geq 1$, $P_i$ is the projective cover of $\ker(d_{i-1})$. 

Every finite-dimensional module $M$ possesses a minimal projective resolution, which is unique up to isomorphism of complexes. The \emph{projective dimension} of $M$, denoted $\pd(M)$, is the length of its minimal projective resolution, that is, the largest integer $n$ such that $P_n \neq 0$ (or infinity if the resolution does not terminate).
For a finite-dimensional algebra $A$, the \emph{global dimension} of $A$, denoted $\gd(A)$, is the supremum of the projective dimensions of all simple $A$-modules:
\[ \gd(A) = \sup \{ \pd(S) \mid S \text{ is a simple } A\text{-module} \} \]
The global dimension is an invariant of the Morita equivalence class of the algebra $A$. If two algebras are Morita equivalent, their global dimensions are equal.
Moreover, it is a known fact for finite-dimensional algebras that the global dimension is bounded above by the length of the longest path in the quiver of the algebra.

\subsection{Lifting projective resolutions of EI-categories}
Let $\E$ be a category and $G$ be a finite group. For simplicity, assume that $\E$ is a subcategory of $\Set$ so that there is a natural wreath product $G \wr \E$. Clearly, every $\mathbb{C}\E$-module $M$ can be inflated to a $\mathbb{C}(G \wr \E)$-module via the action $(f, m) \bullet v = m \bullet v$. We denote the inflation of $M$ to $\mathbb{C}(G \wr \E)$ by $\Inf(M)$. In fact, we consider $\Inf$ as a functor between the category of $\mathbb{C}\E$-modules and $\mathbb{C}(G\wr \E)$-modules. In this subsection, we show that if $\E$ is an EI-category the inflation functor lifts the minimal projective resolution of $M$ to the minimal projective resolution of $\Inf(M)$.

Let $\E$ be an EI-category. For every object $e \in \E^0$, we denote by $H_e$ the associated endomorphism group, and let $P_e$ be a complete set of orthogonal primitive idempotents of $\mathbb{C}H_e$. It is known (see \cite[Lemma 9.31]{Luck1989} or \cite[Corollary 4.5]{Webb2007}) that we can obtain a complete set of orthogonal primitive idempotents $P$ of $\mathbb{C}\E$ by taking the union of all complete sets of orthogonal primitive idempotents of the endomorphism groups:
\[ P = \bigcup_{e \in \E^0} P_e. \]
Therefore, if $V$ is a simple $\mathbb{C}H_e$-module isomorphic to $\mathbb{C}H_e p$ for some $p \in P_e$, then 
\[
\mathbb{C}\E p \cong \mathbb{C}\E e \otimes_{\mathbb{C}H_e} \mathbb{C}H_e p \cong \mathbb{C}\E e \otimes_{\mathbb{C}H_e} V
\]
is an indecomposable projective module of $\mathbb{C}\E$. Moreover, every indecomposable projective module is isomorphic to $\mathbb{C}\E e \otimes_{\mathbb{C}H_e} V$ for some $e \in \E^0$ and $V \in \IRep(\mathbb{C}H_e)$.

It is straightforward to verify that $G \wr \E$ is also an EI-category with the same set of objects $\E^0$. For any object $e \in \E^0$, the associated endomorphism group in the wreath product category is $G \wr H_e$. Given a simple module $V \in \IRep(\mathbb{C}H_e)$, its inflation $\Inf(V)$ is a module over $\mathbb{C}(G \wr H_e)$, which remains simple under the action defined previously. The relationship between the inflation functor and the indecomposable projectives is captured by the following lemma.

\begin{lem} \label[lemma]{lem:Proj_inflation}
Let $e \in \E^0$ and $V \in \IRep(\mathbb{C}H_e)$. Then there is an isomorphism of $\mathbb{C}(G \wr \E)$-modules:
\[
\Inf(\mathbb{C}\E e \otimes_{\mathbb{C}H_e} V) \simeq \mathbb{C}(G \wr \E)e \otimes_{\mathbb{C}(G \wr H_e)} \Inf(V).
\]
\end{lem}
\begin{proof}
We denote by $\mathbf{1}_e$ the function in $G^e$ that maps every element of the set $e$ to the identity $1_G$ and by $\id_e$ we denote the identity function of $e$.
To establish the isomorphism \[\Inf(\mathbb{C}\E e \otimes_{\mathbb{C}H_e} V) \cong \mathbb{C}(G \wr \E)e \otimes_{\mathbb{C}(G \wr H_e)} \Inf(V),\] we first observe that for every $(f,m)\in (G\wr \E)e$ and $v\in V$ we have 
\begin{equation*}
\begin{split}
(f, m) \otimes v &= ((\mathbf{1}_e, m) \cdot (f, \id_e)) \otimes v \\
&= (\mathbf{1}_e, m) \otimes (f, \id_e)\bullet v \\
&= (\mathbf{1}_e, m) \otimes  (\id_e\bullet v) \\
&= (\mathbf{1}_e, m) \otimes v,
\end{split}
\end{equation*}

Next, define a linear map \[\Psi:\Inf(\mathbb{C}\E e \otimes_{\mathbb{C}H_e} V) \to \mathbb{C}(G \wr \E)e \otimes_{\mathbb{C}(G \wr H_e)} \Inf(V).
\] on simple tensors by
\[ \Psi(m \otimes v) = (\mathbf{1}_e, m) \otimes v, \]
where $m \in \mathbb{C}\E e$ and $v \in V$.  To see that $\Psi$ is well-defined, note that for any $h \in H_e$, the inflation action on $V$ satisfies $(\mathbf{1}_e, h) \bullet v = h \bullet v$. Thus,
\begin{equation*}
\begin{split}
\Psi(mh \otimes v) &= (\mathbf{1}_e, mh) \otimes v \\
&= (\mathbf{1}_e, m)(\mathbf{1}_e, h) \otimes v \\
&= (\mathbf{1}_e, m) \otimes (\mathbf{1}_e, h)\bullet v \\
&= (\mathbf{1}_e, m) \otimes h\bullet v \\
&= \Psi(m \otimes h\bullet v),
\end{split}
\end{equation*}which confirms $\Psi$ respects the tensor relation over $\mathbb{C}H_e$.

To verify that $\Psi$ is a $\mathbb{C}(G \wr \E)$-module homomorphism, let $(f, m^\prime) \in G \wr \E$ where $m^\prime \in \E(e', e'')$ and let $m \otimes v \in \mathbb{C}\E e \otimes_{\mathbb{C}H_e} V$ where $m \in \E(e, e')$. By the definition of the inflation action, we have:
\[ \Psi((f, m^\prime) \bullet (m \otimes v)) = \Psi(m^\prime m \otimes v) = (\mathbf{1}_e, m^\prime m) \otimes v. \]
On the other hand, acting on the image gives:
\begin{equation*}
\begin{split}
(f, m^\prime) \bullet \Psi(m \otimes v) &= (f, m^\prime) \bullet ((\mathbf{1}_e, m) \otimes v) \\ 
&= ((f, m^\prime) \cdot (\mathbf{1}_e, m)) \otimes v \\
&= (f \ast m \cdot \mathbf{1}_e, m^\prime m) \otimes v \\
&= (f \ast m, m^\prime m) \otimes v \\
&= (\mathbf{1}_e, m^\prime m) \otimes v.
\end{split}
\end{equation*}

To prove that $\Psi$ is an isomorphism, we construct its inverse 
\[ \Phi: \mathbb{C}(G \wr \E)e \otimes_{\mathbb{C}(G \wr H_e)} \Inf(V) \to \Inf(\mathbb{C}\E e \otimes_{\mathbb{C}H_e} V). \]
Define $\Phi$ on spanning elements by $\Phi((f, m) \otimes v) = m \otimes v$. To show that $\Phi$ is well-defined, let $(g, h) \in G \wr H_e$; then:
\begin{equation*}
\begin{split}
\Phi((f, m) \cdot (g, h) \otimes v) &= \Phi((f\ast h \cdot g, mh) \otimes v) \\
&= mh \otimes v.
\end{split}
\end{equation*}
On the other hand:
\begin{equation*}
\begin{split}
\Phi((f, m) \otimes (g, h)\bullet v) &= \Phi((f, m) \otimes h\bullet v) \\
&= m \otimes h\bullet v \\
&= mh \otimes v.
\end{split}
\end{equation*}
Thus $\Phi$ is well-defined. 

To show that $\Phi$ is a $\mathbb{C}(G \wr \E)$-module homomorphism, let $(g, m^\prime) \in G \wr \E$ and $(f, m) \otimes v$ be a spanning element of the domain. We have:
\begin{equation*}
\begin{split}
\Phi((g, m^\prime) \bullet ((f, m) \otimes v)) &= \Phi(((g, m^\prime) \cdot (f, m)) \otimes v) \\
&= \Phi((g \ast m \cdot f, m^\prime m) \otimes v) \\
&= m^\prime m \otimes v.
\end{split}
\end{equation*}
On the other hand, acting after applying $\Phi$ gives:
\begin{equation*}
\begin{split}
(g, m^\prime) \bullet \Phi((f, m) \otimes v) &= (g, m^\prime) \bullet (m \otimes v) \\
&= m^\prime m \otimes v,
\end{split}
\end{equation*}
so $\Phi$ is a $\mathbb{C}(G \wr \E)$-module homomorphism.

From the definitions it follows immediately that $\Phi \circ \Psi$ is the identity on $\Inf(\mathbb{C}\E e \otimes_{\mathbb{C}H_e} V)$, as:
\[ \Phi(\Psi(m \otimes v)) = \Phi((\mathbf{1}_e, m) \otimes v) = m \otimes v. \]
Conversely, for any $(f, m) \otimes v$ in the target, we previously established the identity $(f, m) \otimes v = (\mathbf{1}_e, m) \otimes v$. Therefore:
\[ \Psi(\Phi((f, m) \otimes v)) = \Psi(m \otimes v) = (\mathbf{1}_e, m) \otimes v = (f, m) \otimes v, \]
which shows $\Psi \circ \Phi$ is the identity. Thus, $\Psi$ is an isomorphism of $\mathbb{C}(G \wr \E)$-modules.
\end{proof}

As a consequence of the fact that inflation preserves the structure of  indecomposable projective modules, we obtain the following result regarding general projective modules.

\begin{cor}
If $P$ is a projective $\mathbb{C}\E$-module, then its inflation $\Inf(P)$ is a projective $\mathbb{C}(G \wr \E)$-module.
\end{cor}

\begin{proof}
Each projective $\mathbb{C}\E$-module $P$ can be decomposed into a direct sum of indecomposable projectives, each isomorphic to $\mathbb{C}\E e \otimes_{\mathbb{C}H_e} V$ for some $e \in \E^0$ and $V \in \IRep(\mathbb{C}H_e)$. According to \Cref{lem:Proj_inflation}, the inflation of each such indecomposable projective is an indecomposable projective $\mathbb{C}(G \wr \E)$-module. Since the inflation functor $\Inf$ preserves direct sums, it follows that $\Inf(P)$ is a direct sum of projective modules and is, therefore, projective.
\end{proof}

The next step is to characterize the radical of an EI-category algebra.

\begin{prop}[{\cite[Proposition 4.6]{Li2011}}]
Let $\E$ be a finite EI-category. The radical $\Rad(\mathbb{C}\E)$ is the subspace spanned by all non-invertible morphisms in $\E$.
\end{prop}

In an EI-category, a morphism $m \in \E(e, e')$ is non-invertible if and only if its domain and codomain are not isomorphic ($e \not\simeq e'$).

\begin{lem}
Let $P$ be a projective $\mathbb{C}\E$-module. Then $\Inf(\Rad P) = \Rad(\Inf P)$.
\end{lem}

\begin{proof}
First, assume $P$ is an indecomposable projective module. Then $P \simeq \mathbb{C}\E e \otimes_{\mathbb{C}H_e} V$ for some $e \in \E^0$ and $V \in \IRep(\mathbb{C}H_e)$. The radical of $P$ is given by $\Rad P = \Rad(\mathbb{C}\E)P = \Rad(\mathbb{C}\E)e \otimes_{\mathbb{C}H_e} V$. Using the characterization of the radical for EI-category algebras, $\Rad P$ and $\Inf(\Rad P)$ are spanned by the set
\[ S = \{ m \otimes v \mid m \in \E(e, e'), \ e' \not\simeq e, \ v \in V \}. \]

On the other hand, by \Cref{lem:Proj_inflation}, $\Inf(P) \simeq \mathbb{C}(G \wr \E)e \otimes_{\mathbb{C}(G \wr H_e)} \Inf V$. The radical of the latter is $\Rad(\mathbb{C}(G \wr \E))e \otimes_{\mathbb{C}(G \wr H_e)} \Inf V$, which is spanned by simple tensors of the form $(f, m) \otimes v$ where $m \in \E(e, e')$ and $e' \not\simeq e$. As shown in the proof of \Cref{lem:Proj_inflation}, $(f, m) \otimes v = (\mathbf{1}_e, m) \otimes v$.

Under the isomorphism $\Phi: \mathbb{C}(G \wr \E)e \otimes_{\mathbb{C}(G \wr H_e)} \Inf V \to \Inf P$, the image of the spanning set 
\[ \{ (\mathbf{1}_e, m) \otimes v \mid m \in \E(e, e'), \ e \not\simeq e', \ v \in V \} \]
is exactly the set $S$ defined above. Therefore, $S$ also spans $\Rad(\Inf P)$. Since both $\Inf(\Rad P)$ and $\Rad(\Inf P)$ are spanned by the same set of elements within $\Inf P$, we have $\Rad(\Inf P) = \Inf(\Rad P)$.

The result for a general projective module $P$ follows immediately from the fact that $P$ is a direct sum of indecomposable projectives and both the radical and the inflation functor preserve direct sums.
\end{proof}

\begin{lem} \label[lemma]{lem:ProjCover_inflation}
Let $M$ be a $\mathbb{C}\E$-module and let $P \xrightarrow{\pi} M$ be its projective cover. Then $\Inf(P) \xrightarrow{\Inf(\pi)} \Inf(M)$ is the projective cover of the $\mathbb{C}(G \wr \E)$-module $\Inf(M)$.
\end{lem}

\begin{proof}
Recall that a surjective homomorphism $\pi: P \to M$ is a projective cover if and only if $P$ is a projective module and $\ker(\pi) \subseteq \Rad(P)$. 

Applying $\Inf$  yields
\[ \Inf(P) \xrightarrow{\Inf(\pi)} \Inf(M), \]
and it is clear that $\Inf(\pi)$ is surjective and $\ker(\Inf(\pi)) \simeq \Inf(\ker(\pi))$.

Next, we know from \Cref{lem:Proj_inflation} that $\Inf(P)$ is a projective $\mathbb{C}(G \wr \E)$-module. Finally, applying the result of our previous lemma, we have:
\[ \ker(\Inf(\pi)) \simeq \Inf(\ker(\pi)) \subseteq \Inf(\Rad P) = \Rad(\Inf P). \]
Since $\Inf(P)$ is projective and the kernel of the surjection is contained in its radical, $\Inf(P) \xrightarrow{\Inf(\pi)} \Inf(M)$ is the projective cover of $\Inf(M)$.
\end{proof}

\begin{cor}
Let $\mathbf{P}_\bullet \to M$ be a minimal projective resolution of a $\mathbb{C}\E$-module $M$. Then the inflated sequence $\Inf(\mathbf{P}_\bullet) \to \Inf(M)$ is a minimal projective resolution of the $\mathbb{C}(G \wr \E)$-module $\Inf(M)$.
\end{cor}

\begin{proof}
Let $\mathbf{P}_\bullet$ be the resolution $\dots \to P_1 \xrightarrow{d_1} P_0 \xrightarrow{\epsilon} M \to 0$. It is easy to verify that $\Inf$ is an exact functor. Therefore, the sequence 
\[ \dots \to \Inf(P_1) \xrightarrow{\Inf(d_1)} \Inf(P_0) \xrightarrow{\Inf(\epsilon)} \Inf(M) \to 0 \]
is exact. By \Cref{lem:Proj_inflation}, each $\Inf(P_i)$ is a projective $\mathbb{C}(G \wr \E)$-module. 

By definition, the resolution $\mathbf{P}_\bullet$ is minimal if $\epsilon$ is a projective cover and each $d_n$ induces a projective cover $P_n \to \ker(d_{n-1})$. By \Cref{lem:ProjCover_inflation}, inflation preserves the projective cover property. Thus, $\Inf(\epsilon)$ is the projective cover of $\Inf(M)$ and each $\Inf(d_n)$ induces the projective cover of $\Inf(\ker(d_{n-1})) \simeq \ker(\Inf(d_{n-1}))$. It follows that $\Inf(\mathbf{P}_\bullet)$ is a minimal projective resolution.
\end{proof}

\begin{cor} \label[corollary]{cor:projective_dimension_and_inflation}
The inflation functor preserves the projective dimension of modules. That is, for any $\mathbb{C}\E$-module $M$, we have
\[ \pd_{\mathbb{C}\E}(M) = \pd_{\mathbb{C}(G \wr \E)}(\Inf M). \]
\end{cor}

\subsection{The Global Dimension}

Finally, we apply the preceding results to compute the global dimension of $\mathbb{C}(G\wr \PT_n)$.
Recall that $\E_n$ is the EI-category whose objects are the subsets of $[n]=\{1,\ldots, n\}$ and whose morphisms are onto functions. Since $\mathbb{C}(G\wr \PT_n)$ is Morita equivalent to the category algebra $\mathbb{C}(G \wr \E_n)$, we may instead determine the global dimension of the latter.

\begin{thm}
The global dimension of the category algebra $\mathbb{C}(G \wr \E_n)$ is $n-1$.
\end{thm}

\begin{proof}
We first establish $n-1$ as an upper bound. The global dimension is bounded above by the length of the longest path in the quiver of the algebra. In \Cref{thm:Quiver_of_GwrE_n}, we characterized the vertices of the quiver as multi-Young diagrams. Specifically, we showed that if $U$ is a multi-Young diagram with $k$ boxes and $V$ is one with $p$ boxes, an arrow $V \to U$ can only exist if $p = k+1$. Furthermore, there are no arrows directed toward the vertex corresponding to a multipartition with 0 boxes (the "bottom" element). Consequently, the longest directed path in the quiver has length $n-1$, which provides the required upper bound: $\gd(\mathbb{C}(G \wr \E_n)) \leq n-1$.

For the lower bound, we consider the $\mathbb{C}\E_n$-simple module $M$ corresponding to the partition $[2, 1^{n-2}]$. It was shown in \cite[Corollary 6.9]{Stein2019} that $\pd_{\mathbb{C}\E_n}(M) = n-1$. By \Cref{cor:projective_dimension_and_inflation}, the inflation functor $\Inf$ preserves projective dimension, implying:
\[ \pd_{\mathbb{C}(G \wr \E_n)}(\Inf M) = \pd_{\mathbb{C}\E_n}(M) = n-1. \]
Since the global dimension of an algebra is the supremum of the projective dimensions of its simple modules, it follows that $\gd(\mathbb{C}(G \wr \E_n)) \geq n-1$. Combining these bounds, we conclude that the global dimension is exactly $n-1$.
\end{proof}

\begin{cor}
For every finite group $G$, the global dimension of the monoid algebra $\mathbb{C}(G\wr \PT_n)$ is $n-1$.
\end{cor}

\section{The quiver of the wreath product of a group with the monoid of all order-preserving partial functions}
\label{sec:quiver_of_PO_n}
A partial function $\alpha \colon [n] \to [n]$ is called \emph{order-preserving} if $x \leq y \implies \alpha(x) \leq \alpha(y)$ for every $x, y$ in the domain of $\alpha$. Let $\PO_n$ be the submonoid of $\PT_n$ consisting of all order-preserving partial functions. In this section, we describe the quiver of the complex algebra $\mathbb{C}(G \wr \PO_n)$. 

In this case, the wreath product is 
\[ G \wr \PO_n = \{ (f, \alpha) \mid f \in \PT([n], G), \, \alpha \in \PO_n, \text{ and } \dom(\alpha) = \dom(f) \}. \]
Recall the set of idempotents $\mathcal{E} = \{ (1_X, \text{id}_X) \mid X \subseteq [n] \}$, which forms a subsemilattice of $G \wr \PO_n$. It is a straightforward consequence that $G \wr \PO_n$ is an $\mathcal{E}$-Ehresmann and a right restriction monoid.

Let $\EO_n$ be the subcategory of the category $\E_n$, defined in \Cref{sec:Ehresmann}, which has the same set of objects but whose morphisms consist of surjective, total order-preserving functions. It follows that $G \wr \EO_n$ is the Ehresmann category associated with $G \wr \PO_n$. Thus, by \Cref{thm:Ehresmann_Iso}, we obtain an isomorphism of algebras 
\[ \mathbb{C}(G \wr \PO_n) \simeq \mathbb{C}(G \wr \EO_n). \]

Following the approach in \Cref{sec:The_Skeleton}, we now consider the skeleton of $G \wr \EO_n$. If $\alpha \colon X \to Y$ is an order-preserving bijection, then its inverse $\alpha^{-1}$ is also order-preserving. Consequently, \Cref{lem:Isomorphis_objects} implies that two objects $X$ and $Y$ of $G \wr \EO_n$ are isomorphic if and only if $|X| = |Y|$. Let $\SEO_n$ denote the category whose objects are $[k]$ for $0 \leq k \leq n$, and whose morphisms are surjective, total order-preserving functions. The skeleton of $G \wr \EO_n$ is then the wreath product $G \wr \SEO_n$. Our goal now is to describe the quiver of the algebra $\mathbb{C}(G \wr \SEO_n)$.

The endomorphism groups in $G \wr \SEO_n$ have a straightforward structure. Since the only order-preserving bijection from $[k]$ to itself is the identity map, the endomorphism group of any object $[k]$ in $G \wr \SEO_n$ is simply the direct product $G^k$.

By applying \Cref{thm:QuiverOfEICategories} to the skeletal category $G \wr \SEO_n$, the set of vertices of $Q$ is the disjoint union of the simple modules of the endomorphism groups of the objects in the skeleton. Since the endomorphism group of the object $[k]$ is isomorphic to $G^k$, the vertex set is:
\[ \bigsqcup_{k=0}^{n} \IRep(G^k). \]
Any simple module of $G^k$ is an outer tensor product of $k$ simple modules of $G$. Thus, the vertices of the quiver are naturally indexed by sequences $(V_1, \dots, V_k)$ of length $0 \le k \le n$, where each $V_i \in \IRep(G)$. So the vertex set is:
\[ \bigsqcup_{k=0}^{n} \{ V_1 \boxtimes \dots \boxtimes V_k \mid V_i \in \IRep(G) \}. \]
For $k=0$, the unique vertex is the trivial module of the trivial group $G^0$.

Our next step is to identify the irreducible morphisms in the category $G \wr \SEO_n$. 

\begin{lem} \label[lemma]{lem:IrreducibleMorphisms_Order}
The irreducible morphisms of $G \wr \SEO_n$ are precisely the morphisms from $[k+1]$ to $[k]$ for $0 \leq k < n$. That is,
\[
\Irr(G \wr \SEO_n)([r],[k]) = 
\begin{cases} 
G \wr \SEO_n([r],[k]) & \text{if } r = k+1, \\
\emptyset & \text{otherwise.}
\end{cases}
\]
\end{lem}

\begin{proof}
As in the proof of \Cref{lem:IrreducibleMorphisms}, any morphism from $[k+1]$ to $[k]$ is clearly irreducible because any factorization would force one of the factors to be an isomorphism.
Conversely, suppose $(f,\alpha) \in G \wr \SEO_n([r],[k])$ with $r > k+1$. Since $\alpha$ is a surjective order-preserving function, it is known that $\alpha$ factors as $\alpha = \alpha_{1}\alpha_{2}$, where $\alpha_{2} \colon [r] \to [k+1]$ and $\alpha_{1} \colon [k+1] \to [k]$ are both surjective order-preserving functions (see \cite[Lemma 5.1]{Stein2016}). We may then define the morphisms $(\mathbf{1}_{[k+1]},\alpha_{1})$ and $(f,\alpha_{2})$. Since neither $\alpha_1$ nor $\alpha_2$ is a bijection, these factors are not invertible in $G \wr \SEO_n$. Their product is
\[
(\mathbf{1}_{[k+1]},\alpha_{1}) \cdot (f,\alpha_{2}) = (\mathbf{1}_{[k+1]}\ast\alpha_{2} \cdot f, \alpha_{1}\alpha_{2}) = (f, \alpha),
\]
which shows that $(f, \alpha)$ is not irreducible.
\end{proof}

 Let $V \in \IRep(G^{p})$ and $U \in \IRep(G^k)$ be vertices of the quiver. If $p \neq k+1$, there are no arrows in the quiver of $\mathbb{C}(G \wr \SEO_n)$ from $V$ to $U$ because $\Irr(G \wr \SEO_n)([p],[k])$ is empty. Consequently, we focus on the case $p = k+1$ and examine the structure of $\mathbb{C}[\Irr(G \wr \SEO_n)([k+1], [k])]$ as a $(G^k \times G^{k+1})$-module, with the action defined in \Cref{thm:QuiverOfEICategories}.

A morphism in $\Irr(G \wr \SEO_n)([k+1], [k])$ is a pair $(f, \alpha)$, where $\alpha \colon [k+1] \to [k]$ is a surjective order-preserving function and $f \colon [k+1] \to G$.
It is easy to see that there are exactly $k$ surjective order-preserving maps from $[k+1]$ to $[k]$, denoted $\sigma_1, \dots, \sigma_k$. The map $\sigma_i$ is defined by:
\[ 
\sigma_i(j) = 
\begin{cases} 
j & \text{if } j \leq i, \\
j-1 & \text{if } j > i.
\end{cases} 
\]
so $\alpha=\sigma_i$ for some $i$.
Set $M = \mathbb{C}[\Irr(G \wr \SEO_n)([k+1], [k])]$ and define \[M_i = \text{span}\{ (f, \sigma_i) \mid f \in G^{[k+1]} \}.\]
We naturally identify $G^{[k]}$ with $G^k$. 
\begin{lem} \label[lemma]{lem:ModuleDecomposition}
The $(G^k \times G^{k+1})$-module $M$  decomposes as a direct sum of $k$ submodules:
\[ M \cong \bigoplus_{i=1}^k M_i, \]
\end{lem}

\begin{proof}
Since the set of all such pairs $(f, \alpha)$ forms a basis for $M$, and since the $M_i$ are disjoint except for the zero vector and their union spans $M$, we can write $M$ as a direct sum of vector spaces \[ M \cong \bigoplus_{i=1}^k M_i. \]
To show this is a decomposition of $(G^k \times G^{k+1})$-modules, we examine the action defined in \Cref{thm:QuiverOfEICategories}. For $h \in G^k$ and $g \in G^{k+1}$, the action on a basis element $(f, \sigma_i)$ is:
\[ (h, g) \bullet (f, \sigma_i) = (h , \text{id}_{[k]}) \cdot (f, \sigma_i) \cdot (g, \text{id}_{[k+1]})^{-1}. \]
Since $(g, \text{id}_{[k+1]})^{-1}=(g^{-1}, \text{id}_{[k+1]})$ and by calculating the product in the wreath product category:
\[  (h , \text{id}_{[k]}) \cdot (f, \sigma_i) \cdot (g^{-1}, \text{id}_{[k+1]}) = ((h\ast \sigma_i)\cdot f\cdot g^{-1} , \sigma_i). \]

Crucially, the underlying order-preserving map $\sigma_i$ remains unchanged by the action of the endomorphism groups so each subspace $M_i$ is a submodule of $M$.
\end{proof}

For each $1 \leq i \leq k$, let $X_i = \{ (f, \sigma_i) \mid f \in G^{[k+1]} \}$ denote the natural basis for the subspace $M_i$. By \Cref{lem:ModuleDecomposition}, the action of $G^k \times G^{k+1}$ preserves $M_i$ and maps basis elements to basis elements. Consequently, we can view $M_i$ as the permutation module $\mathbb{C}X_i$ arising from the action of $G^k \times G^{k+1}$ on the set $X_i$. To analyze the structure of this permutation module, we first establish that this action is transitive.

\begin{lem} \label[lemma]{lem:TransitiveAction2}
For each $1 \leq i \leq k$, the action of $G^k \times G^{k+1}$ on the set $X_i$ is transitive.
\end{lem}

\begin{proof}
Let $(\mathbf{1}_{[k+1]}, \sigma_i)$ be the element of $X_i$ where $\mathbf{1}_{[k+1]}$ is the constant function mapping every element of $[k+1]$ to $1_G$. For any arbitrary $(f, \sigma_i) \in X_i$, we choose $(\mathbf{1}_{[k]}, f^{-1}) \in G^k \times G^{k+1}$. Applying the action, we have:
\[
\begin{aligned}
(\mathbf{1}_{[k]}, f^{-1}) \bullet (\mathbf{1}_{[k+1]}, \sigma_i) &= (\mathbf{1}_{[k]}, \text{id}_{[k]}) \cdot (\mathbf{1}_{[k+1]}, \sigma_i) \cdot (f^{-1}, \text{id}_{[k+1]})^{-1} \\
&= (\mathbf{1}_{[k]}, \text{id}_{[k]}) \cdot (\mathbf{1}_{[k+1]}, \sigma_i) \cdot (f, \text{id}_{[k+1]}) \\
&= ((\mathbf{1}_{[k]} \ast \sigma_i) \cdot \mathbf{1}_{[k+1]} \cdot f, \sigma_i ) \\
&= (\mathbf{1}_{[k+1]} \cdot \mathbf{1}_{[k+1]} \cdot f, \sigma_i) \\
&= (f, \sigma_i).
\end{aligned}
\]
This shows that any element of $X_i$ can be reached from $(\mathbf{1}_{[k+1]}, \sigma_i)$, and thus the action is transitive.
\end{proof}

The transitivity of the action allows us to identify each $M_i$ as a permutation module induced by the stabilizer of our chosen base point. Let $K_i$ denote the stabilizer of $(\mathbf{1}_{[k+1]}, \sigma_i)$ in $G^k \times G^{k+1}$.

\begin{lem} \label[lemma]{lem:Stabilizer}
For each $1 \le i \le k$, the stabilizer $K_i$ of the element $(\mathbf{1}_{[k+1]}, \sigma_i)$ is given by
\[ K_i = \{ (h, h \ast \sigma_i) \mid h \in G^k \}. \]
\end{lem}

\begin{proof}
An element $(h, g) \in G^k \times G^{k+1}$ belongs to $K_i$ if and only if $(h, g) \bullet (\mathbf{1}_{[k+1]}, \sigma_i) = (\mathbf{1}_{[k+1]}, \sigma_i)$. This condition is satisfied if and only if:
\[ ((h \ast \sigma_i) \cdot \mathbf{1}_{[k+1]} \cdot g^{-1}, \sigma_i) = (\mathbf{1}_{[k+1]}, \sigma_i). \]
Equating the first components, we obtain the equation $(h \ast \sigma_i) \cdot g^{-1} = \mathbf{1}_{[k+1]}$, which implies $g = h \ast \sigma_i$. 
\end{proof}

In particular, the elements of $K_i$ are in one-to-one correspondence with $G^k$. Note that if we write $h \in G^k$ as $h = (h_1, \dots, h_k)$, then \[h \ast \sigma_i = (h_1, \dots, h_i, h_i, h_{i+1}, \dots, h_k)\].

Since $M_i$ is a transitive permutation module, we have the following isomorphism of $(G^k \times G^{k+1})$-modules:
\[ M_i \cong \text{Ind}_{K_i}^{G^k \times G^{k+1}}(\tr_{K_i}). \]

For each $1 \leq i \leq k$, we define a group monomorphism $\varphi_i \colon G^k \to G^{k+1}$ by $\varphi_i(h) = h \ast \sigma_i$. Explicitly, for $h = (h_1, \dots, h_k) \in G^k$, we have
\[ \varphi_i(h_1, \dots, h_k) = (h_1, \dots, h_i, h_i, h_{i+1}, \dots, h_k). \]

It is easy to verify that this is indeed a group monomorphism.
\begin{prop} \label[proposition]{prop:MultiplicityInduction}
Let $U \in \IRep(G^k)$ and $V \in \IRep(G^{k+1})$. For each $1 \leq i \leq k$, the multiplicity of the simple $(G^k \times G^{k+1})$-module $U \boxtimes V^*$ in $M_i$ is equal to the multiplicity of $V$ in the induced module $\text{Ind}_{G^k}^{G^{k+1}}(U)$, where the induction is taken along the group homomorphism $\varphi_i \colon G^k \to G^{k+1}$. In terms of characters, this is expressed as:
\[ \langle U \boxtimes V^*,M_i \rangle_{G^k \times G^{k+1}} = \langle V,\Ind_{G^k}^{G^{k+1}} U \rangle_{G^{k+1}}. \]
\end{prop}

\begin{proof}
Using the character of the permutation module $M_i$ as identified in \Cref{lem:Stabilizer}:
\[ \langle U \boxtimes V^*, M_i \rangle_{G^k \times G^{k+1}} = \langle U \boxtimes V^*, \text{Ind}_{K_i}^{G^k \times G^{k+1}} \tr_{K_i} \rangle_{G^k \times G^{k+1}}. \]
Applying Frobenius reciprocity:
\[ \langle U \boxtimes V^*, \text{Ind}_{K_i}^{G^k \times G^{k+1}} \tr_{K_i} \rangle_{G^k \times G^{k+1}}=\langle \text{Res}_{K_i}^{G^k \times G^{k+1}} (U \boxtimes V^*), \tr_{K_i} \rangle_{K_i}. \]
As a summation, we have:
\[
\begin{aligned}
\langle \text{Res}_{K_i}^{G^k \times G^{k+1}} (U \boxtimes V^*), \tr_{K_i} \rangle_{K_i} &=  \frac{1}{|G^k|} \sum_{h \in G^k} (U \boxtimes V^*)(h, h\ast \sigma_i) \cdot 1\\
&=\frac{1}{|G^k|} \sum_{h \in G^k} (U(h) \cdot V^*(h\ast \sigma_i))\\
&=\frac{1}{|G^k|} \sum_{h \in G^k} (U(h) \cdot \overline{V(h\ast \sigma_i)})
\end{aligned}
\]
Note that $V(h\ast \sigma_i)=\Res_{G^k}^{G^{k+1}}V(h)$ where the restriction is taken along the homomorphism $\varphi_i$.

Therefore, the final expression becomes:
\[
\frac{1}{|G^k|} \sum_{h \in G^k} U(h) \cdot \overline{\Res_{G^k}^{G^{k+1}}V(h)} = \langle U, \Res_{G^k}^{G^{k+1}} V \rangle_{G^k}.
\]
Finally, applying Frobenius reciprocity again, we obtain:
\[
\langle U, \Res_{G^k}^{G^{k+1}} V \rangle_{G^k} = \langle \Ind_{G^k}^{G^{k+1}} U, V \rangle_{G^{k+1}}.
\]
This completes the proof.
\end{proof}
Clearly, since $\varphi_i$ acts as the identity on all coordinates $j \neq i$, investigating the induction from $G^k$ to $G^{k+1}$ along $\varphi_i$ reduces to understanding the induction from $G$ to $G \times G$ along the diagonal homomorphism $d(g)=(g,g)$.

\begin{lem} \label[lemma]{lem:TensorInduction}
Let $U_1, U_2, U_3 \in \IRep(G)$. The multiplicity of $U_1 \boxtimes U_2$ in the induced module $\Ind_G^{G \times G} U_3$ is equal to the multiplicity of $U_3$ in the tensor product $U_1 \otimes U_2$ as a $G$-module. 
\end{lem}

\begin{proof}
By Frobenius reciprocity for the group pair $(G, G \times G)$, we have:
\[ \langle \Ind_G^{G \times G} U_3, U_1 \boxtimes U_2 \rangle_{G \times G} = \langle U_3, \Res_G^{G \times G} (U_1 \boxtimes U_2) \rangle_G. \]
The restriction of the external tensor product $U_1 \boxtimes U_2$ to the diagonal subgroup is the internal tensor product $U_1 \otimes U_2$. Substituting this into the inner product, we obtain:
\[ \langle U_3, \text{Res}_G^{G \times G} (U_1 \boxtimes U_2) \rangle_G = \langle U_3, U_1 \otimes U_2 \rangle_G. \]
\end{proof}

\begin{lem} \label[lemma]{lem:FinalMultiplicity}
Let $U = U_1 \boxtimes \dots \boxtimes U_k \in \text{IRep}(G^k)$ and $V = V_1 \boxtimes \dots \boxtimes V_{k+1} \in \text{IRep}(G^{k+1})$. 
The multiplicity of $U \boxtimes V^*$ in $M_i$ is non-zero only if:
\begin{itemize}
    \item $V_r \cong U_r$ for all $r < i$,
    \item $V_r \cong U_{r-1}$ for all $r > i+1$.
\end{itemize}
In this case, the multiplicity is given by the multiplicity of $U_i$ in the tensor product $V_i\otimes V_{i+1}$.
\end{lem}

\begin{proof}
From \Cref{prop:MultiplicityInduction}, we have:
\[ \langle M_i, U \boxtimes V^* \rangle = \langle \Ind_{G^k}^{G^{k+1}} U, V \rangle_{G^{k+1}}. \]
Since the induction along $\varphi_i$ is the identity on all components except the $i$-th one, where it is the diagonal induction, we have:
\[ \Ind_{G^k}^{G^{k+1}} U \cong U_1 \boxtimes \dots \boxtimes U_{i-1} \boxtimes \Ind_G^{G \times G}(U_i) \boxtimes U_{i+1} \boxtimes \dots \boxtimes U_k. \]
Comparing this with $V = V_1 \boxtimes \dots \boxtimes V_{k+1}$, the result follows from \Cref{lem:TensorInduction}.
\end{proof}

If we use the Kronecker delta notation:
\[ \delta_{U,V} = \begin{cases} 1 & \text{if } U \cong V \\ 0 & \text{otherwise} \end{cases} \]
then the multiplicity formula can be expressed compactly as:
\[ \langle M_i, U \boxtimes V^* \rangle = \left( \prod_{r=1}^{i-1} \delta_{U_r, V_r} \right) \left( \prod_{r=i+1}^{k} \delta_{U_r, V_{r+1}} \right) \langle U_i, V_i \otimes V_{i+1} \rangle_G. \]
In order to obtain the number of arrows from $V$ to $U$ all is left is to sum this number over $i$ from $1$ to $k$.
We can conclude:

\begin{thm} \label{thm:QuiverConstruction}
The quiver $Q$ of the algebra $\mathbb{C}(G \wr \SEO_n)$ or, equivalently, $\mathbb{C}(G \wr \PO_n)$ is constructed as follows:
\begin{itemize}
    \item \textbf{Vertices:} The set of vertices is the disjoint union $\displaystyle \bigsqcup_{k = 0}^n \IRep(G^k)$, where each vertex is a simple module of the form $U = U_1 \boxtimes \dots \boxtimes U_k$ with $U_r \in \IRep(G)$.
    \item \textbf{Arrows:} For any $V=V_1 \boxtimes \dots \boxtimes V_{k+1} \in \IRep(G^{k+1})$ and $U=U_1 \boxtimes \dots \boxtimes U_k \in \IRep(G^k)$, the number of arrows from $V$ to $U$ is given by the sum:
    \[ \sum_{i=1}^k \left( \prod_{r=1}^{i-1} \delta_{U_r, V_r} \right) \left( \prod_{r=i+1}^{k} \delta_{U_r, V_{r+1}} \right) \langle U_i, V_i \otimes V_{i+1} \rangle_G, \]
    where $\langle U_i, V_i \otimes V_{i+1} \rangle_G$ is the multiplicity of $U_i$ in the tensor product $V_i \otimes V_{i+1}$ as a $G$-module.
\end{itemize}
\end{thm}

\begin{exmp} \label[example]{ex:Z2Quiver}
We illustrate \Cref{thm:QuiverConstruction} by constructing the quiver $Q$ of the algebra $\mathbb{C}(C_2 \wr \PO_3)$. The group $G = C_2$ has two simple modules: the trivial module $\tr$ and the sign module $\sgn$. 

The tensor product rules for these modules are straightforward:
\begin{center}
\begin{tabular}{c|cc}
$\otimes$ & $\tr$ & $\sgn$ \\ \hline
$\tr$ & $\tr$ & $\sgn$ \\
$\sgn$ & $\sgn$ & $\tr$
\end{tabular}
\end{center}
The quiver is a bit crowded, and given by:
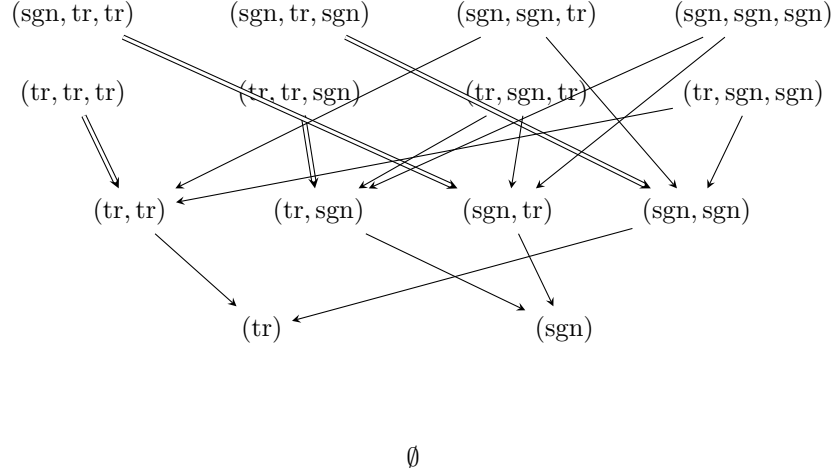
\begin{center}
\noindent
\begin{tikzpicture}[yscale=1.3, xscale=1.0, >=stealth]
    \node (L3e) at (0,4.5) {$(\sgn, \tr, \tr)$};
    \node (L3f) at (3,4.5) {$(\sgn, \tr, \sgn)$};
    \node (L3g) at (6,4.5) {$(\sgn, \sgn, \tr)$};
    \node (L3h) at (9,4.5) {$(\sgn, \sgn, \sgn)$};
    
    \node (L3a) at (0,3.7) {$(\tr, \tr, \tr)$};
    \node (L3b) at (3,3.7) {$(\tr, \tr, \sgn)$};
    \node (L3c) at (6,3.7) {$(\tr, \sgn, \tr)$};
    \node (L3d) at (9,3.7) {$(\tr, \sgn, \sgn)$};

    \node (L2a) at (0.75,2.5) {$(\tr, \tr)$};
    \node (L2b) at (3.25,2.5) {$(\tr, \sgn)$};
    \node (L2c) at (5.75,2.5) {$(\sgn, \tr)$};
    \node (L2d) at (8.25,2.5) {$(\sgn, \sgn)$};

    \node (L1a) at (2.5,1.3) {$(\tr)$};
    \node (L1b) at (6.5,1.3) {$(\sgn)$};

    \node (L0) at (4.5,0) {$\emptyset$};

    \draw[->] (L2a) -- (L1a);
    \draw[->] (L2b) -- (L1b);
    \draw[->] (L2c) -- (L1b);
    \draw[->] (L2d) -- (L1a);

    \draw[->, double, double distance=1.2pt] (L3a) -- (L2a);
    \draw[->, double, double distance=1.2pt] (L3b) -- (L2b);
    \draw[->, double, double distance=1.2pt] (L3e) -- (L2c);
    \draw[->, double, double distance=1.2pt] (L3f) -- (L2d);

    \draw[->] (L3c) -- (L2c);
    \draw[->] (L3c) -- (L2b);
    \draw[->] (L3d) -- (L2d);
    \draw[->] (L3d) -- (L2a);
    \draw[->] (L3g) -- (L2a);
    \draw[->] (L3g) -- (L2d);
    \draw[->] (L3h) -- (L2b);
    \draw[->] (L3h) -- (L2c);

\end{tikzpicture}
\end{center}
Note that a multiplicity of two between vertices is denoted by a double-lined arrow.
\end{exmp}

\bibliographystyle{plain}
\bibliography{library}
\end{document}